\documentclass[12pt]{amsart}
\usepackage[dvipsnames]{xcolor}
\usepackage[english]{babel}
\usepackage{amsfonts}
\usepackage{amsthm}
\usepackage{amsbsy}
\usepackage{amssymb}
\usepackage{comment}
\usepackage{graphicx}
\usepackage{eso-pic}
\usepackage{tikz} \usetikzlibrary{arrows,decorations.markings,matrix}
\usepackage{xfrac}
\usepackage{float}
\usepackage[bookmarks=true, colorlinks=true, urlcolor=
black, linkcolor=
black,citecolor=
black, pagebackref=true]{hyperref}
\usepackage{bookmark}
\usepackage{amsmath, amscd}
\usepackage{subfigure}
\usepackage{scalerel,stackengine}
\usepackage{epigraph}
\usepackage{placeins}
\usepackage[marginparwidth=2cm]{geometry}
\usepackage{nicefrac}
\usepackage{todonotes}
\usepackage[normalem]{ulem}

\newtheorem{lemma}{Lemma}[section]
\newtheorem{thm}[lemma]{Theorem}
\newtheorem*{thm*}{Theorem}
\newtheorem{prop}[lemma]{Proposition}
\newtheorem{cor}[lemma]{Corollary}
\newtheorem*{cor*}{Corollary}

\theoremstyle{definition}
\newtheorem{defn}[lemma]{Definition}

\theoremstyle{remark}
\newtheorem{rem}[lemma]{Remark}

\newcounter{todocounter}

\newcommand{\sm}{\smallsetminus}
\renewcommand{\setminus}{\smallsetminus}

\newcommand{\matR}{\mathbb{R}}
\newcommand{\matQ}{\mathbb{Q}}
\newcommand{\matZ}{\mathbb{Z}}
\newcommand{\matC}{\mathbb{C}}
\newcommand{\matP}{\mathbb{P}}
\newcommand{\matH}{\mathbb{H}}
\newcommand{\matS}{\mathbb{S}}
\newcommand{\matE}{\mathbb{E}}
\newcommand{\matRP}{\mathbb{RP}}
\newcommand{\matCP}{\mathbb{CP}}

\newcommand{\matK}{\mathbb{K}}
\newcommand{\matO}{\mathbb{O}}
\newcommand{\matF}{\mathbb{F}}
\newcommand{\matL}{\mathbb{L}}

\newcommand{\calR}{\mathcal{R}}

\newcommand{\calL}{\mathcal{L}}

\newcommand{\calP}{\mathcal{P}}

\newcommand{\SO}{\mathrm{SO}}

\newcommand{\SL}{\mathrm{SL}}
\newcommand{\GL}{\mathrm{GL}}
\newcommand{\PGL}{\mathrm{PGL}}
\newcommand{\PU}{\mathrm{PU}}

\newcommand{\Or}{\mathrm{O}}

\newcommand{\Isom}{\mathrm{Isom}}

\newcommand{\Vol}{\mathrm{Vol}}
\newcommand{\PSL}{\mathrm{PSL}}

\newcommand{\id}{\mathrm{id}}

\newcommand{\Aut}{\mathrm{Aut}}

\newcommand{\e}{\mathrm{e}}

\newcommand{\w}[1]{{\color{white} #1}}

\makeatletter
\newcommand\xleftrightarrow[2][]{\ext@arrow 0099{\longleftrightarrowfill@}{#1}{#2}}
\def\longleftrightarrowfill@{\arrowfill@\leftarrow\relbar\rightarrow}
\makeatother

\author{Stefano Riolo}
\address{Dipartimento di Matematica, Universit\`a di Bologna \newline
         Piazza di Porta San Donato 5, 40126 Bologna, Italy}
\email{stefano\w{0}.\w{0}riolo\w{0}@unibo.it}
\urladdr{\href{https://www.dm.unibo.it/~stefano.riolo}{www.dm.unibo.it/~stefano.riolo}}

\author{Andrea Seppi}
\address{Dipartimento di Matematica ``Giuseppe~Peano'', Universit\`a di Torino \newline Via Carlo Alberto 10, 10123 Torino, Italy}
\email{andrea\w{0}.\w{0}seppi\w{0}@unito.it}
\urladdr{\href{https://andreaseppi.github.io}{andreaseppi.github.io}}

\author{Leone Slavich}
\address{Dipartimento di Matematica (DIMA), Universit\`a di Genova \newline Via Dodecaneso 35, 16146 Genova, Italy}
\email{leone\w{0}.\w{0}slavich\w{0}@unige.it}
\urladdr{\href{https://leoneslavich.github.io/}{leoneslavich.github.io}}

\title[Convex projective manifolds, symmetric spaces and decompositions]{Convex projective manifolds, symmetric spaces and geometric decompositions}

\begin{document}

\begin{abstract}
We prove that if a closed, 
indecomposable, properly convex real projective 4-manifold is geometric or admits a geometric decomposition in the sense of Thurston, then every piece  is real hyperbolic. This extends a theorem of Benoist to dimension four. Moreover, we build orientable (non-hyperbolic) 4-manifolds of the above type, with arbitrary positive even Euler characteristic. 
Finally, we characterise the  compact locally symmetric spaces that virtually support properly convex real projective structures in terms of their geometry. 
\end{abstract}

\maketitle

\section{Introduction} 
This work addresses some geometric and topological aspects of
closed manifolds admitting a convex projective structure. These aspherical manifolds, which include hyperbolic manifolds and tori as special examples, share some remarkable analogies with non-positively curved 
manifolds. 

A \emph{convex projective manifold} is the quotient of a properly convex open set $\Omega \subset \matRP^n$ by a discrete torsion-free group $\Gamma < \PGL_{n+1}(\matR)$ preserving $\Omega$. It is \emph{indecomposable} if, roughly speaking, 
$\Omega$ is not the convex hull of some disjoint lower-dimensional domains
. A \emph{convex projective structure} on a manifold $M$ is a diffeomorphism between $M$ and a convex projective manifold $\Omega/\Gamma$. For simplicity, unless otherwise stated, all convex projective manifolds in this introduction are orientable.

The topology of these manifolds is quite well understood in dimension up to three. Indeed, a closed surface of genus $g$ admits a convex projective (indecomposable) structure if and only if $g \geq 1$ (resp. $g \geq 2$). Moreover, Benoist \cite{B4} proved that a closed, 
indecomposable, convex projective 3-manifold has only hyperbolic pieces in its JSJ decomposition.
We thus focus on manifolds of dimension at least four, and more precisely on two classes: four-manifolds that are geometric or admit a geometric decomposition in the sense of Thurston, and locally symmetric manifolds of arbitrary dimension. 
%

\subsection{Four-dimensional geometries and decompositions} \label{sec:intro-decompositions}

In four dimensions, we extend Benoist's theorem as follows:

\begin{thm} \label{thm:geom-4d} 
    If a closed, indecomposable, convex projective $4$-manifold admits a (possibly trivial) geometric decomposition, then all the pieces are real hyperbolic.
\end{thm}

A manifold $M$ is said to admit a \emph{geometric decomposition} if it can be cut along a (possibly disconnected
) hypersurface into pieces whose interior is \emph{geometric}, that is admits a complete, finite-volume, locally homogeneous, Riemannian metric. The decomposition is \emph{trivial} if the hypersurface is empty; in other words, $M$ is geometric.

Theorem \ref{thm:geom-4d} is new even 
in the geometric case. Indeed, Benoist \cite[Problem 4]{Bquest} proposed to prove that the fundamental group $\Gamma$ of a closed convex projective manifold $M = \Omega/\Gamma$ is not isomorphic to any lattice in $\PU(2,1)$. Equivalently \cite{C}, no divisible properly convex domain $\Omega$ with its Hilbert metric is quasi-isometric to the complex hyperbolic space $\matH^2(\matC)$. We provide a simple solution to this problem, not even relying on the convexity of the projective structure: since, as shown by Avez and Kobayashi \cite{A,Kob}, the rational Pontryagin classes of any $(\PGL_{n+1}(\matR), \matRP^n)$-manifold are trivial, the 4-manifold $M$ has signature 
$\sigma = 0$ by the Hirzebruch signature theorem. Closed complex hyperbolic surfaces are instead well known to have $\sigma > 0$. 


Admitting a geometric decomposition is a very restrictive 
requirement for a 4-manifold \cite{Hpaper}. Such hypothesis cannot be removed in Theorem \ref{thm:geom-4d}, as there are closed, indecomposable, convex projective $n$-manifolds, $n\geq4$, that do not satisfy it. Among them, there are Benoist's examples \cite{Bqi}, many Gromov--Thurston manifolds provided by Kapovich \cite{K}, some Dehn fillings of cusped hyperbolic manifolds by Choi, Lee and Marquis \cite{CLM2} (see also \cite{LMR}), and the manifolds more recently built by Blayac and Viaggi \cite{BV}. Indeed, if an $n$-manifold $M$ admits a (possibly trivial) geometric decomposition into real hyperbolic pieces, then $\pi_1(M)$ is hyperbolic relatively to a (possibly empty) collection of abelian subgroups of rank $n-1$ (
the cusp groups). The manifolds of the latter two classes do not have this property, while the former have hyperbolic $\pi_1$ but are not locally symmetric (the first ones by \cite{Bqi} combined with Theorem \ref{thm:geom-4d}). 

On the other hand, for all $n \leq 8$, there exist closed convex projective $n$-manifolds admitting non-trivial geometric decompositions with real hyperbolic pieces
. This is shown in Ballas--Danciger--Lee \cite{BDL} and Blayac--Viaggi \cite{BV} for $n=3$, and follows from the work of Choi--Lee--Marquis \cite{CLM1} on convex projective Coxeter truncation polytopes for $n \geq 4$
. One of these polytopes has a role in this paper; see Section \ref{sec:intro-geography} below.

\vspace{-0.04cm}

\subsection{Non-positively curved symmetric spaces} \label{sec:intro-symmetric} Thurston geometries are classified in dimension up to 5 \cite{F,Geng,Tbook}, and our proof of Theorem \ref{thm:geom-4d} relies on the classification \cite{F} (see Section \ref{sec:proof-thm-1}). Our approach does not extend directly in dimension 5, not even in the geometric case.
On the other hand, it is natural to restrict the attention to the  subclass of non-positively curved (equivalently, aspherical) locally symmetric manifolds
. We obtain the following characterisation:

\begin{thm}\label{teo:NPC-locally-symmetric-intro}
    Let $Y$ be a symmetric space and $\Lambda < \Isom(Y)$ be a cocompact lattice. The following are equivalent:
    \begin{itemize}
        \item $\Lambda$ has a torsion-free subgroup $\Lambda'$ of finite index such that the manifold $M = Y /\Lambda'$ admits a convex projective structure;
        \item the de Rham decomposition of $Y$ 
        is $Y \cong X_1 \times  \dots \times X_r \times \matE^d$, where $r \leq d+1$ and each $X_i$ is the symmetric space of non-compact type associated to one of the simple Lie groups $\mathrm{SO}^+_{m,1}(\mathbb{R})$, 
        $\mathrm{SL}_m(\matR)$, $\mathrm{SL}_m(\matC)$, $\mathrm{SL}_m(\matH)$, or $E_6^{-26}$.
    \end{itemize}
\end{thm}
(Here $\matE^d$ denotes the Euclidean space of dimension $d$ as a symmetric space. The exceptional Lie group $E_6^{-26}$, in Helgason's notation \cite{Helgason}, may be seen as a sort of $\SL_3$ over the octonions.) The convex projective structure is indecomposable exactly when $d=0$, i.e. the symmetric space $Y = X_1$ is irreducible (and of non-compact type).

Theorem \ref{teo:NPC-locally-symmetric-intro} contains as a special case the solution to Benoist's \cite[Problem 4]{Bquest}. It characterises the compact locally symmetric spaces that virtually admit a convex projective structure in terms of the geometry of their universal cover $Y$:
this happens precisely when $Y$ has a properly convex projective model (see Section \ref{sec:projective-models}). As proved by Vinberg \cite{V} in the 1960s, in the case $d=0$, the symmetric space $Y = X_1$ has a convex projective model precisely when it is associated to one of the simple Lie groups of Theorem \ref{teo:NPC-locally-symmetric-intro}. 

The proof of Theorem \ref{teo:NPC-locally-symmetric-intro} carries more information than the sole statement reproduced here. Loosely speaking, it shows that all convex projective structures on compact locally symmetric spaces are built via a standard procedure from some basic ingredients. These consist of real hyperbolic lattices that divide indecomposable, non-symmetric, properly convex domains, and irreducible lattices in semisimple Lie groups that act cocompactly on a product of indecomposable, symmetric, properly convex domains. These actions are then ``combined'' with a representation of $\mathbb{Z}^d<\mathbb{R}^d$ in the group of independent positive homotheties in the cones above these domains (together with possible extra factors coming from cones over points). We refer the reader to Section \ref{sec:spazi-simmetrici} (in particular to Theorem \ref{teo:ostruzioni-NPC-localmente-simmetrici}) for the details.


\subsection{Euler characteristic of four-manifolds} \label{sec:intro-geography}

Let us now come back to dimension four. Two important homotopy invariants for 
4-manifolds are the Euler characteristic $\chi$ and the signature $\sigma$ of the intersection form. Given a class of closed oriented 4-manifolds, the \emph{geography problem} consists in determining the pairs $(\chi,\sigma) \in \matZ^2$ that are realised by manifolds in that class. This is a classical problem in several contexts, like complex surfaces, symplectic manifolds, Kähler manifolds, Einstein manifolds, manifolds of positive scalar curvature, and many others. 

Let $M$ be a closed oriented convex projective $4k$-manifold. As we already observed $\sigma(M) = 0$, so by the Poincar\'e duality $\chi(M)$ is even
. This reduces the geography problem for convex projective 4-manifolds to realising or excluding some even integers as the Euler characteristics of such manifolds.
We build closed 
convex projective 4-manifolds of any possible positive Euler characteristic:

\begin{thm} \label{thm:construction} 
    Every positive even integer is the Euler characteristic of a closed, orientable, convex projective $4$-manifold.
\end{thm}

The examples we build do admit a non-trivial geometric decomposition, so they provide a new class of manifolds satisfying the conditions of Theorem \ref{thm:geom-4d}. An essential ingredient is a convex projective reflection group found by Choi, Lee and Marquis \cite{CLM1}.

The sometimes called ``Hopf--Thurston conjecture'' predicts that any closed aspherical $2m$-manifold $M$ satisfies $(-1)^m \cdot \chi(M) \geq 0$. Closed, decomposable, convex projective manifolds, like tori, have $\chi = 0$ \cite[Fact 2.15]{LMR}. So, {if the conjecture were true} in dimension four
, Theorem \ref{thm:construction} would completely determine the spectrum of the Euler characteristics of closed orientable convex projective 4-manifolds
: the even natural numbers. 

Let us briefly mention what was previously known in regard. {For non-positively curved locally symmetric spaces the conjecture is well known to hold true. In the real hyperbolic case, we have indeed $\chi(M^{2m}) = (-1)^m C_m \cdot \Vol(M^{2m})$ with $C_m > 0$ by the Chern--Gauss--Bonnet theorem}. The volume spectrum (equivalently, the actual values of $\chi$) of closed hyperbolic 4-manifolds is currently unknown. The smallest known examples have $\chi = 16$ \cite{CM,L}. The unique non-hyperbolic closed convex projective 4-manifold with previously known $\chi > 0$ has smaller $\chi = 12$ \cite{LMR}. It has been found by Dehn filling a cusped hyperbolic manifold following \cite{CLM2,KS}. (These manifolds cover non-orientable hyperbolic and convex projective manifolds with $\chi = 8$ and $6$, respectively.) 

\subsection{The proof of Theorem \ref{thm:geom-4d}} \label{sec:proof-thm-1}

While in three dimensions there are 8 Thurston geometries, the four-dimensional ones are infinitely many, often referred to as 18 plus a countable infinite family, as proven by Filipkiewicz \cite{F}. Our proof {of Theorem \ref{thm:geom-4d}} relies on this classification.

There are two situations to consider: when the geometric decomposition of $M$ is trivial, and when it is not. 
In the former case, a strong restriction on the possible geometries is given by excluding the non-aspherical and the solvable geometries. The latter are excluded by a result of Islam and Zimmer \cite{IZ}, asserting that if $\Omega/\Gamma$ is  compact (or more generally convex cocompact), then every virtually solvable subgroup of $\Gamma$ is virtually abelian. Since our $M = \Omega/\Gamma$ is indecomposable, $\Gamma$ is not virtually abelian, and therefore not solvable by Islam and Zimmer's result. 
In order to exclude the {
geometries $\matH^3 \times \matE^1$, $\widetilde{\matS\matL} \times \matE^1$ and $\matH^2 \times \matE^2$,} we apply a result of Wang \cite[Corollary 8.28]{Rag} 
that implies that the centre of $\Gamma$ is infinite. This contradicts the fact (proven by Benoist \cite{B2} and Vey \cite{Vey}) that the centre of a group that divides an indecomposable domain is trivial. 
At the end, we are left with $\matH^4$, $\matH^2(\matC)$ and $\matH^2\times\matH^2$, and the latter two geometries are excluded by Theorem \ref{teo:NPC-locally-symmetric-intro}.

The analysis of the situation when the geometric decomposition is non-trivial is more involved, and we do not give a complete overview here. It relies on previous work of Hillman \cite{Hpaper} concerning the possible geometries of the pieces of a 4-manifold with a geometric decomposition. The above result of Islam and Zimmer is now applied to the cusp subgroups of each piece. Then, we are essentially left with excluding the geometries $\matH^3 \times \matE^1$, $\widetilde{\matS\matL} \times \matE^1$, $\matH^2 \times \matE^2$ and 
$\matH^2\times\matH^2$, for which key ingredients are an important result of Islam and Zimmer \cite{IZ} on the centralisers of abelian subgroups of convex cocompact groups, and the irreducibility of the action of $\pi_1(M)$ on $\SL_5(\matR)$ proven by Benoist \cite{B2} and Vey \cite{Vey}. 

\subsection{The proof of Theorem \ref{teo:NPC-locally-symmetric-intro}} \label{sec:proof-thm-2}

When the symmetric space $Y$ has real rank one, Theorem \ref{teo:NPC-locally-symmetric-intro} states that $Y/\Lambda$ virtually admits a convex projective structure if and only if $Y$ is the real hyperbolic space. In higher dimension $n>4$, this is a consequence of the vanishing of the rational Pontryagin classes, together with Novikov's homeomorphism invariance of the rational Pontryagin classes \cite{N} and Farrell and Jones's topological rigidity \cite{FJ1,FJ2}. See Proposition \ref{prop:rank-1} for the details. 

The most technical part of Theorem \ref{teo:NPC-locally-symmetric-intro} then concerns the higher rank case, and consists of two parts. The proof that the second item implies the first is constructive. When $d=0$, it is a consequence of Vinberg's classification of homogeneous convex cones \cite{V} (see also Koecher \cite{Koe}), and the convex projective structures that one constructs are of the form $\Omega/\Lambda'$ for $\Omega$ a symmetric domain, namely the projective model of the symmetric space of $Y$, with the only exception of $Y=\matH^n$, in which case $\Omega$ may not be symmetric. When $d>0$, the construction additionally relies on Eberlein's work \cite{EB1,EB2,EB3} on non-positively curved locally symmetric spaces and Vey's characterisation \cite{B2,Vey} of the centre of a group that divides a properly convex cone. 

The proof of the other implication in fact shows, roughly speaking, that every convex projective structure on $Y/\Lambda'$ arises as one constructed in the proof of the first implication. In addition to the tools mentioned above, the proof significantly applies Benoist's Zariski density theorem \cite{B2} and Margulis' superrigidity \cite{M}. For simplicity, we assume here $d=0$, and we give an idea of how these results are combined. 

If the lattice $\Lambda'$ in the identity component $G$ of $\Isom(Y)$ is isomorphic to a discrete group of projective transformations that acts cocompactly on a properly convex domain $\Omega$, there is a group homomorphism $\phi \colon \Lambda' \to \PGL_{n+1}(\matR)$ with discrete image. If moreover $\Omega$ is not symmetric, the image is Zariski dense by Benoist's density theorem \cite{B2}. We are thus in a good setup to apply superrigidity. If $G$ has rank at least two, $\phi$ extends to a Lie group isomorphism $G \to \PGL_{n+1}(\matR)$, which immediately gives a contradiction. The situation when $G$ has rank one was already discussed at the beginning of this section (where actually it is only the complex hyperbolic case that crucially uses the vanishing of the Pontryagin classes, while $G=\mathrm{SL}_m(\matH)$ or $E_6^{-26}$ could alternatively be excluded by Corlette's superrigidity similarly to the higher rank case). 

Although based on similar ideas, the proof of the other implication for $d>0$ is more technical, and crucially uses Vey's characterisation of the centre of groups dividing domains and, again, Eberlein’s results \cite{EB1, EB2, EB3}.

\subsection{The proof of Theorem \ref{thm:construction}} \label{sec:proof-thm-3}

It is less difficult to build cusped hyperbolic manifolds than closed ones. Ratcliffe and Tschantz \cite{RT} determined the whole volume spectrum of cusped hyperbolic 4-manifolds as follows. By orbifold covering a specific Coxeter polytope, they explicitly built some cusped hyperbolic manifolds with minimal $\chi = 1$. Since some of these manifolds have Betti number $b_1 > 0$, they are cyclically covered by manifolds with arbitrary $\chi > 0$. Our approach to proving Theorem \ref{thm:construction} is similar, and moreover relies on a cusped hyperbolic manifold as well. We indeed build a closed convex projective 4-manifold $M$ with $\chi(M)=2$ (that is the minimal $\chi > 0$) and $b_1(M) > 0$, by orbifold covering a Coxeter 4-polytope $P$. The condition on $b_1$ is {immediate thanks} to the fact that, by construction, $M$ admits a non-trivial geometric decomposition with a single hyperbolic piece $M_1$.

For this purpose, we consider a specific hyperbolic Coxeter simplex $S = \matH^4/\Gamma_S$ with one ideal vertex. By the work of Choi, Lee and Marquis \cite{CLM1} relying on Vinberg's theory of linear reflection groups \cite{V3}, the ideal vertex of $S$ can be, roughly speaking,  truncated orthogonally, in such a way to get a compact convex projective Coxeter prism $P = \Omega/\Gamma_P${. The groups $\Gamma_P, \Gamma_S < \PGL_5(\matR)$ and the domains $\Omega, \matH^4 \subset \matRP^4$ are related by continuous deformation}. Crucially, the orbifold Euler characteristic of $P$ and $S$ equals $1/960$. To prove Theorem \ref{thm:construction}, we thus find a torsion-free subgroup $\Gamma < \Gamma_P$ of index $1920$, so that the convex projective manifold $M = \Omega/\Gamma$ has $\chi(M) = 1920/960 = 2$.

However, we could not find such a subgroup purely computationally, not even with opportune computer software. 
Finding normal subgroups is less difficult than arbitrary subgroups, but unfortunately $\Gamma_P$  does not have such a normal subgroup.
The presence of the cusped hyperbolic simplex $S$  is essential for our proof. Indeed, we first find a torsion-free normal subgroup $\Gamma_1 < \Gamma_S$ of index $1920$, and so a cusped hyperbolic manifold $M_1 = \matH^4 / \Gamma_1$ (the piece) with $\chi(M_1) = 2$.

Replacing $S$ with $P$, we get a convex projective 4-manifold $\overline{M}_1$ with 
totally geodesic boundary. Actually, it is an orbifold $\overline{M}_1 = \Omega/\overline{\Gamma}_1$ with mirror boundary that covers $P$, topologically obtained from $M_1$ by truncating its 
cusps. The torsion in $\overline{\Gamma}_1$ comes from the truncation. To get the closed convex projective manifold $M$, we glue the boundary components of $\overline{M}_1$ in pairs. It is crucial that $\Gamma_1$ (and so $\overline{\Gamma}_1$) is normal in $\Gamma_S$ (resp. $\Gamma_P$) and that $S$ has a unique ideal vertex. Indeed, unlike hyperbolic geometry, in projective geometry there is no canonical reflection associated to a hyperplane. So, it does not suffice to choose isometric gluings to ensure that the resulting closed manifold admits a convex projective structure. To pair the boundary components of $\overline{M}_1$, we use automorphisms of the regular orbifold covering $\overline{M}_1 \to P$.

\subsection{Related work and questions}

Theorem \ref{thm:geom-4d} can be considered as a generalisation of Benoist's theorem on 3-manifolds \cite{B4}, not only because the hypothesis is empty in dimension $n \leq 3$ and necessary for $n > 3$. 
On the one hand, our proof works in lower dimensions as well (see Remark \ref{rem:low-dim-geom-dec}). On the other hand, one may object that Benoist's theorem includes some additional geometric information: the JSJ tori are totally geodesic convex projective submanifolds, each the quotient of the interior of a properly embedded projective triangle in $\Omega$. But the analogous fact in our case already follows from our proof: essentially, from Islam and Zimmer's centraliser theorem \cite{IZ}. 
Alternatively, this additional information can be directly deduced from the statement of Theorem \ref{thm:geom-4d} by the flat torus theorem for convex projective manifolds of Islam and Zimmer \cite{IZ2}, of which their centraliser theorem is somehow an extension. 

Even more, by a theorem of Bobb \cite{B}, every closed indecomposable convex projective $n$-manifold $M$ contains a canonical (possibly empty) finite collection of embedded, disjoint, flat, virtual $(n-1)$-tori, consisting precisely of the image of the properly embedded (interiors of the) projective $(n-1)$-simplices in the domain. If $M$ admits a geometric decomposition into real hyperbolic pieces, 
by the flat torus theorem, one concludes that Bobb's virtual tori are precisely the ones of the geometric decomposition. 

We add a couple of observations. First we note that, roughly speaking, the convex projective manifolds from \cite{BV,CLM2,CLM1,K,LMR} (and, 
partially, also the ones from \cite{Bqi}) 
mentioned in Section \ref{sec:intro-decompositions} are somehow built by manipulating an auxiliary hyperbolic manifold: topologically by surgery, geometrically by deformation. 
It would be interesting to know whether there are non-symmetric convex projective manifolds 
of different nature.

Second, as mentioned above, in this work we made use of the result of Avez and Kobayashi on the vanishing of the rational Pontryagin classes of real projective manifolds, allowing us to exclude complex hyperbolic geometry. 
As evidence that this fact has not been fully exploited so far
, we refer to Section \ref{sec:Ontaneda} for another straightforward application, that answers a question of Ruffoni \cite[Question 1]{R}.

We conclude the introduction with some questions naturally arising from this work:
\begin{enumerate}
    \item If a closed, indecomposable, convex projective manifold of dimension $>4$ admits a (possibly trivial) geometric decomposition, are all the pieces real hyperbolic? 
    \item Is every positive even integer the Euler characteristic of a closed, orientable, convex projective 4-manifold of the types \cite{Bqi,BV,CLM2,K} mentioned above?
    \item Is every odd natural number the Euler characteristic of a closed, non-orientable, convex projective 4-manifold?
\end{enumerate}

\subsection*{Structure of the paper} In Section \ref{sec:rank1} we prove the rank-one part of Theorem \ref{teo:NPC-locally-symmetric-intro}. Section \ref{sec:prel} contains some foundational theorems on convex projective manifolds. In Section \ref{sec:spazi-simmetrici} we conclude the proof of Theorem \ref{teo:NPC-locally-symmetric-intro}, in Section \ref{sec:4d} we prove Theorem \ref{thm:geom-4d}, and in Section \ref{sec:chi} we prove Theorem \ref{thm:construction}.

\subsection*{Acknowledgments}
We are particularly grateful to Lorenzo Ruffoni, for raising our attention on wheth\-er complex hyperbolic surfaces may admit convex projective structures. S.R. thanks Mauricio Bustamante, Stefano Francaviglia, Mitul Islam, Gye-Seon Lee, Ludovic Marquis, Marco Moraschini, Beatrice Pozzetti, Grazia Rago, and J\'er\'emy Toulisse for useful conversations. A.S. is grateful to Davide Dameno, Colin Davalo, Martin Deraux, Philippe Eyssidieux, Balthazar Fl\'echelles, Fanny Kassel and Pierre Py for enlightening discussions. 

We thank Daniel Fox and Andrea Tamburelli for helping us to find out that the vanishing of the rational Pontryagin classes of real projective manifolds, of which we initially intended to include a proof in the special case of convex projective manifolds, has already been established by Avez and Kobayashi.

{\footnotesize \subsection*{Funding information.} S.R. and L.S. were funded by the European Union – NextGenerationEU under the National Recovery and Resilience Plan (PNRR) -- Mission 4 Education and research -- Component 2 From research to business -- Investment 1.1 Notice Prin 2022 -- DD N. 104 del 2/2/2022, with title ``Geometry and topology of manifolds'', proposal code 2022NMPLT8 -- CUP J53D23003820001. A.S. was funded by the European Union (ERC, GENERATE, 101124349). Views and opinions expressed are however those of the author(s) only and do not necessarily reflect those of the European Union or the European Research Council Executive Agency. Neither the European Union nor the granting authority can be held responsible for them.}

\section{Topology and real rank one} \label{sec:rank1}

The main goal of this section is to prove the rank-one part of Theorem \ref{teo:NPC-locally-symmetric-intro}. This is done in Section \ref{sec:rank-1}. Before that, in Section \ref{sec:prel-0} we introduce some basic definitions and facts about convex projective manifolds, and in Section \ref{sec:pontryagin} we deal with some topological obstructions for a manifold to admit a projective structure. Finally, Section \ref{sec:Ontaneda} is a parenthesis on a question of Ruffoni.

\subsection{Convex projective manifolds} \label{sec:prel-0}

A \emph{properly convex domain} is an open set $\Omega \subset \matRP^n$ whose closure is contained in an affine chart, in which $\Omega$ is convex. Any such $\Omega$ has its canonical Hilbert metric, and the group $\mathrm{Aut}(\Omega)=\{ g \in \PGL_{n+1}(\matR) \, | \, g(\Omega) = \Omega \}$ acts on $\Omega$ by isometry. Therefore, a subgroup $\Gamma<\mathrm{Aut}(\Omega)$ acts properly discontinuously on $\Omega$ if and only if it is discrete.

If a discrete subgroup $\Gamma < \Aut(\Omega)$ acts cocompactly on a properly convex domain $\Omega$, we say that $\Gamma$ \emph{divides} $\Omega$ and that $\Omega$ is \emph{divisible}. The action of $\Gamma$ on $\Omega$ is free if and only if $\Gamma$ is torsion free. By Selberg's lemma, this condition can always be achieved \emph{virtually}, i.e. up to possibly passing to a finite-index subgroup of $\Gamma$.

A \emph{convex projective manifold} is a smooth manifold of the form $\Omega/\Gamma$, where $\Omega$ is a properly convex domain and $\Gamma<\mathrm{Aut}(\Omega)$ is discrete and torsion free. More generally, to deal with orbifolds, it suffices to remove the torsion-free condition. A \emph{convex projective structure} on a smooth manifold $M$ is a diffeomorphism from $M$ to a convex projective manifold $\Omega/\Gamma$. This is equivalent to a \emph{projective structure} on $M$ --- i.e. a $(G,X)$-structure with $X = \matRP^n$ and $G=\PGL_{n+1}(\matR)$ --- whose developing map is a diffeomorphism with $\Omega$. Note that $\Omega \to \Omega/\Gamma \cong M$ is a universal cover and $\Gamma \cong \pi_1(M)$. In particular, the universal cover of $M$ is contractible and therefore $M$ is aspherical. 

Recall that Whitehead's theorem implies that two aspherical manifolds $M$ and $N$ are homotopy equivalent if $\pi_1(M) \cong \pi_1(N)$. A closed manifold $M$ is \emph{topologically rigid} if every homotopy equivalence between $M$ and a topological manifold $N$ is homotopic to a homeomorphism. Closed surfaces and closed aspherical 3-manifolds are well known to be topologically rigid
. By a result of Farrell and Jones \cite[Corollary 0.3]{FJ2}, closed convex projective $n$-manifolds are topologically rigid for $n > 4$ as well. Therefore, at least in dimension $n \neq 4$, closed convex projective $n$-manifolds are determined up to homeomorphism by their fundamental group.


\subsection{Topology} \label{sec:pontryagin}

Our starting point, Theorem \ref{thm:pontryagin} below, is a topological obstruction for (not necessarily compact, nor convex) projective manifolds of arbitrary dimension. It has been provided by Kobayashi \cite[Theorem 4.16]{Kob} in the 1980s, relying on a previous result of Avez \cite{A}. We first recall a definition.

Given a complex vector bundle $E \to M$, let $c_j(E) \in H^{2j}(M; \matQ)$ denote the image of its $j$-th Chern class in rational cohomology. Assume now that $M$ is a smooth manifold. The $k$-th \emph{rational Pontryagin class} of $M$ is then defined by
$$p_k(M) = (-1)^k c_{2k}(TM\otimes_\matR \matC) \in H^{4k}(M; \matQ).$$
(Observe that $c_{2k+1}(TM\otimes_\matR \matC)$ vanishes, so the rational Chern classes are meaningful only in even degree.) By a famous theorem of Novikov \cite{N}, these classes are homeomorphism invariant for closed oriented smooth manifolds.

\begin{thm}[Avez \cite{A}, Kobayashi \cite{Kob}] \label{thm:pontryagin}
    If $M$ admits a projective structure, then $p_k(M) = 0$ for all $k>0$.
\end{thm}

Recall that the \emph{intersection form} of a closed oriented manifold $M$ of dimension $n=4k$ is the symmetric, non-degenerate, bilinear form
$$H^{2k}(M; \matZ) \times H^{2k}(M; \matZ) \to \matZ$$
defined via the cup product $\smile$ and the Poincar\'e duality by $\alpha \cdot \beta = \langle \alpha \smile \beta, [M] \rangle$, where $[M]$ is the fundamental class of $M$. Let $b^+(M)$ and $b^-(M)$ denote the number of positive and negative eigenvalues of the associated real form, respectively. Then the \emph{signature} of $M$ is defined by $\sigma(M) = b^+(M) - b^-(M) \in \matZ$.
    
\begin{cor} \label{cor:chi-sigma}
    If a closed oriented manifold $M$ of dimension $n \equiv 0 \mod 4$ admits a projective structure, then $\sigma(M) = 0$. In particular, $\chi(M)$ is even. 
\end{cor}

\begin{proof}
    Recall that, for each partition $\iota = (i_1, \ldots, i_m)$ such that $i_1 + \ldots + i_m = n$, the $\iota$-th \emph{Pontryagin number} of $M$ is 
    $p_\iota(M) = \langle p_{i_1}(M) \smile \ldots \smile p_{i_m}(M), [M] \rangle \in \matZ$. By the Hirzebruch signature theorem, for every positive integer $k$ there exists a linear combination $\ell_k$ of the Pontryagin numbers such that $\sigma(M) = \ell_k(M)$ for every closed oriented $4k$-manifold $M$. As in our case $p_\iota(M) = 0$ for each $\iota$ by Theorem \ref{thm:pontryagin}, we have $\sigma(M) = 0$. 

    In particular, $\chi(M)$ is even because $\chi(M)$ and $\sigma(M)$ have always the same parity. Indeed, denoting by $b_i$ the $i$-th Betti number, by definition and Poincar\'e duality
    \begin{align*}
        \chi(M)  & = \sum_{i=0}^{4k} (-1)^i b_i(M) = b_{2k}(M) + 2 \sum_{i=0}^{2k-1} b_i(M) \\
        & \equiv b_{2k}(M) = b^+(M) + b^-(M) & \mod2& \\
        & \equiv b^+(M) - b^-(M) = \sigma(M) & \mod2&.
    \end{align*}
\end{proof}

\subsection{Rank-one locally symmetric spaces} \label{sec:rank-1}

Recall that a non-positively curved symmetric space $X$ of real rank one is a $\matK$-hyperbolic $m$-space for $\matK = \matR, \matC, \matH, \matO$, where $m = \dim_\matK(X) \geq 2$. In the octonionic case $\matK = \matO$, we have only $m=2$. The previous obstructions allow us to generalise and solve Benoist's \cite[Problem 4]{Bquest}:

\begin{prop} \label{prop:rank-1} 
    If a group that divides a properly convex domain is isomorphic to a $\matK$-hyperbolic lattice, then $\matK = \matR$.
\end{prop}

\begin{proof}
Let $\Gamma$ divide $\Omega \subset \matRP^n$, where $n \geq 4$. By contradiction, we suppose that $\Gamma$ is isomorphic to a $\matK$-hyperbolic lattice $\Lambda$ with $\matK \neq \matR$. Up to replacing with finite-index subgroups, we assume by the Selberg lemma that $\Gamma \cong \Lambda$ are torsion free and that the manifolds $M = \Omega/\Gamma$ and $N = \matH^m(\matK)/\Lambda$ are orientable.

We claim that $M$ and $N$ have the same real dimension $n$ and are both closed (in particular, $\Lambda$ is cocompact). By asphericity, $M$ and $N$ are homotopically equivalent.  Now, assume by contradiction that $N$ is non-compact, which implies that the cohomological dimension of $\Lambda$ is $<d$. By the Margulis lemma, $N$ is diffeomorphic to the interior of a compact $d$-manifold $\overline{N}$ with $\pi_1$-injective, aspherical, boundary components. Therefore $\Lambda \cong \pi_1(\overline{N})$ has cohomological dimension $d-1$ (because the cohomological dimension of a group is at least the one of a subgroup), and so $n = d-1$ because $\Gamma \cong \Lambda$. But $n$ is odd (because $\matK \neq \matR$, so $d$ is even), so $\chi(M) = 0$ by the Poincar\'e duality. This contradicts the (well-known) fact that $\chi(N)>0$ by the Chern--Gauss--Bonnet theorem, as $\chi$ is a homotopy invariant. (Alternatively, as $\matK \neq \matR$, the cusp groups in $\Lambda$ are well-known to be virtually solvable but not virtually abelian, and this contradicts Islam and Zimmer's Theorem \ref{thm:IZ}--(1), stated and used in the sequel.) Therefore $N$ is closed and $d = n$. 

If $n=4$, so $\matK = \matC$ and $m = 2$, we have a contradiction because the signature is a homotopy invariant. Indeed, Corollary \ref{cor:chi-sigma} gives $\sigma(M)=0$, while $\sigma(N) \neq 0$ for every 
closed complex hyperbolic surface $N$. The latter (well-known) fact follows, for example, by combining the Hirzebruch signature theorem with the Hirzebruch proportionality principle
, since the compact dual of $\matH^2(\matC)$ is $\matCP^2$ and $\sigma(\matCP^2) \neq 0$. (In fact, the same argument gives a contradiction in higher dimensions for $\matK = \matC, \matH, \matO$ when $m$ is even.) 

If instead $n > 4$, by the topological rigidity of $M$ \cite{FJ2} (or of $N$ \cite{FJ1}
), the manifolds $M$ and $N$ are homeomorphic. Therefore, by Novikov's theorem \cite{N}, they have the same rational Pontryagin classes, which are trivial by Theorem \ref{thm:pontryagin}. But the only closed hyperbolic manifolds with trivial rational Pontryagin classes are the real ones \cite{T}, and we have another contradiction.
\end{proof}

An equivalent statement follows for the sake of completeness.
\begin{cor}\label{cor:no complex hyperbolic}
    If a divisible convex domain is quasi-isometric to a $\matK$-hyperbolic space, then $\matK = \matR$.
\end{cor} 

\begin{proof}
Let $\Gamma$ divide 
$\Omega$ and assume that $\Omega$ is quasi-isometric to $\matH^m(\matK)$. By the Milnor--\v{S}varc Lemma, $\Gamma$ is quasi-isometric to $\Omega$ (endowed with any $\Gamma$-invariant metric), and to $\matH^m(\matK)$. By Chow and Pansu's quasi-isometric rigidity \cite{C,P}, $\Gamma$ admits a finite-index subgroup isomorphic to a lattice in $\Isom(\matH^m(\matK))$, so Proposition \ref{prop:rank-1} applies.    
\end{proof}

We conclude the section with a couple of observations.

\begin{rem} 
The argument of Proposition \ref{prop:rank-1} does not apply to all the remaining locally symmetric spaces of non-compact type that do not admit convex projective models. The ones associated to complex Lie groups have indeed trivial rational Pontryagin classes (see Tshishiku \cite{T}). Thus the proof of Theorem \ref{teo:NPC-locally-symmetric-intro}, that will be done in Section \ref{sec:spazi-simmetrici}, needs to rely on a different approach.
\end{rem}

\begin{rem} \label{rem:no-corlette} In the proof of Theorem \ref{teo:NPC-locally-symmetric-intro}, Proposition \ref{prop:rank-1} is really essential only to exclude complex hyperbolic lattices.
Indeed, the higher-rank lattices will be excluded in Proposition \ref{prop:localmente-simmetrico} thanks to Benoist's density Theorem \ref{teo:Benoist} and Margulis's superrigidity. The quaternionic and octonionic hyperbolic lattices can (and will) also be excluded in the same way by Corlette's superrigidity (see Theorem \ref{teo:superrigidity}). 
\end{rem}

\subsection{Ontaneda's manifolds}\label{sec:Ontaneda} 

Before moving on, we remark that Ruffoni \cite[Question 1]{R} asked whether there exists a closed aspherical manifold $M$ such that (1) $M$ does not admit any real projective structure, (2) $M$ does not admit any flat conformal structure, and (3) $\pi_1(M)$ is hyperbolic.

Proposition \ref{prop:rank-1} gives a simple, positive, answer. Indeed, the rational Pontryagin classes of locally conformally flat manifolds are known to be trivial. As a consequence, complex hyperbolic manifolds do not admit any locally conformally flat structure, because (as already mentioned in the proof of Proposition \ref{prop:rank-1}) their rational Pontryagin classes are not all trivial \cite{T}. Hence complex hyperbolic manifolds satisfy all the three conditions (1), (2) and (3) above. This seems to indicate that the result on the triviality of Pontryagin classes has not been taken into account  in several contexts. 

In the spirit of Ruffoni's question, it is now natural to add the extra requirement that (4) $\pi_1(M)$ is not isomorphic to any lattice in a semisimple Lie group of real rank one. We get the same, simple, positive answer with another class of manifolds: the negatively-pinched manifolds with non-trivial rational Pontryagin classes built by Ontaneda via Riemannian hyperbolisation \cite[Corollary 4]{O}.
By (1), these are also examples of manifolds of arbitrary dimension $n \geq 4$ whose fundamental groups do not admit any projective Anosov representation into $\SL_d(\matR)$ for $d \leq n+1$, although they do for some $d$ by \cite{DFWZ,LR2}.

\section{Preliminaries on convex projective manifolds} \label{sec:prel}

Before continuing with the proof of Theorem \ref{teo:NPC-locally-symmetric-intro}, we collect in this section some foundational results on divisible convex domains and convex projective manifolds. 

\subsection{(In)decomposable domains} \label{sec:(in)dec}

We call \emph{convex cone} an open subset of $\matR^{n+1}$ invariant under linear combinations with positive coefficients. A convex cone is \emph{proper} if its closure is contained in a half-space of the form $a_0x_0+\dots+a_nx_n>0$. 

Given a convex cone $C \subset \matR^{n+1}$, we set $$\Aut(C)=\{g\in \mathrm{GL}_{n+1}(\matR)\,|\,g(C)=C\}.$$
Let $\Omega_1=\mathbb{P}(C_1)$ and $\Omega_2=\mathbb{P}(C_1)$ be properly convex domains obtained by projectivising two proper convex cones $C_1 \subset V_1$ and $C_2 \subset V_2$ respectively (where $V_1$ and $V_2$ are finite dimensional real vector spaces). The direct sum $C_1 \oplus C_2$ is a convex cone in $V_1 \oplus V_2$, and we define the (projective) \emph{join} of $\Omega_1$ and $\Omega_2$ as 
$$\Omega_1 \ast \Omega_2=\mathbb{P}(C_1\oplus C_2)\subset \mathbb{P}(V_1 \oplus V_2).$$ 
{If $\dim(\Omega_i) = n_i$, then $\Omega_1 \ast \Omega_2$ is a properly convex domain of dimension $n_1 + n_2 + 1$.}

A properly convex domain $\Omega =\mathbb{P}(C) \subset \matRP^n$ is said to be \emph{decomposable} if there exists a non-trivial direct sum decomposition $\matR^{n+1}=V_1 \oplus V_2$ which induces a decomposition of $C$, in the sense that $C=(C\cap V_1) \oplus (C \cap V_2)$. In this case $\Omega=\Omega_1 \ast \Omega_2$ with $\Omega_i=\mathbb{P}(C \cap V_i)$. The domain $\Omega$ is said to be \emph{indecomposable} otherwise.

Any properly convex domain $\Omega=\mathbb{P}(C)\subset \matRP^n$ can be written as a join of indecomposable domains 
\begin{equation}\label{eq:indecomposable-deomposition}\Omega=\Omega_1 \ast \dots \ast \Omega_k,\end{equation}
meaning that there exists a direct sum decomposition $\matR^{n+1}=V_1 \oplus \dots \oplus V_k$ such that $\Omega_i=\mathbb{P}(C\cap V_i)$ is indecomposable for every $i=1,\dots,k$. Let us set $C_i= C \cap V_i$ and let $\matR_{+}^k=\{A \in \mathrm{GL}_{n+1}(\matR) \, |\, \forall i=1,\dots,k\; \exists \lambda_i >0 :\, Av=\lambda_i v\; \forall v \in V_i\}$ denote the group of ``independent'' positive homotheties in each factor $V_i$. Then $\mathrm{Aut}(C_1) \times \dots \times \mathrm{Aut}(C_k)$ is a finite-index subgroup of $\mathrm{Aut}(C)$ and

\begin{equation}\label{eq:automorphism-divisible}
\mathbb{P}(\mathrm{Aut}(C_1) \times \dots\times \mathrm{Aut}(C_k))\cong \mathrm{Aut}(\Omega_1) \times \dots \times \mathrm{Aut}(\Omega_k) \times \mathbb{P}(\mathbb{\matR}_{+}^k)
\end{equation} is a finite-index subgroup of $\mathrm{Aut}(\Omega)$. The full group $\mathrm{Aut}(\Omega)$ acts via non-trivial permutations on the projectively equivalent factors of \eqref{eq:indecomposable-deomposition} (if they exist).

\subsection{Lifts to the special and general linear groups} \label{sec:lifts}

We start with an easy property for a subgroup of the projective automorphisms of properly convex domains and cones.

\begin{defn}\label{defn:divide-cone}
    Let $C \subset \matR^{n+1}$ be a convex cone. A group $\widetilde{\Gamma}<\mathrm{Aut}(C)$ \emph{spans} the cone $C$ if there exists a compact set $K \subset C$ such that $\widetilde{\Gamma} \cdot K=C$. If $\widetilde{\Gamma} < \mathrm{Aut}(C)$ is discrete and spans the cone $C$, then $\widetilde{\Gamma}$ \emph{divides} and $C$ is said to be \emph{divisible}.
\end{defn}

We denote, for a subgroup $G<\GL_{n+1}(\matR)$ and $C$ a convex cone in $\matR^{n+1}$,
$$\Aut_G(C)=\{g\in G\,|\,g(C)= C\}.$$
We will particularly apply this notation to $G=\SL^\pm_{n+1}(\matR)$, where
$$\SL^\pm_{n+1}(\matR)=\{g\in \GL_{n+1}(\matR)\,|\,\det(g)=\pm 1\}~.$$

Now, let $\Omega$ be a properly convex domain in $\matRP^{n}$, and let $\pi \colon \matR^{n+1}\setminus \{0\}\to\matRP^n$ be the projection. Then $\pi^{-1}(\Omega)$ has two connected components, each of which is a proper convex cone in $\matR^{n+1}$. We can lift any group $\Gamma < \Aut(\Omega)$ to $\SL^\pm_{n+1}(\matR)$ as follows:

\begin{prop}\label{prop:lift}
    
    Let $\Omega \subset \matRP^{n}$ be a properly convex domain and $\Gamma < \Aut(\Omega)$. Then there exists a subgroup $\widehat \Gamma<\SL^\pm_{n+1}(\matR)$ preserving each connected component of $\pi^{-1}(\Omega)$ that lifts $\Gamma$, that is, such that the projection $\GL_{n+1}(\matR)\to\PGL_{n+1}(\matR)$ induces an isomorphism between $\widehat\Gamma$ and $\Gamma$.
\end{prop}
\begin{proof}
    Given $\gamma \in \Gamma$, let $M_\gamma\in \GL_{n+1}(\matR)$ be a representative of $\gamma$. Then $M_\gamma$ either preserves the two connected components of $\pi^{-1}(\Omega)$, or switches them, so up to composing with $-\mathrm{id}$ if necessary, we can assume that $M_\gamma$ preserves each connected component of $\pi^{-1}(\Omega)$. Dividing by $\sqrt[n+1]{|\mathrm{det}(M_\gamma)|}$, we can find a representative that moreover has determinant $\pm 1$. This defines a group homomorphism from $\Gamma$ to $\SL^\pm_{n+1}(\matR)$, which is clearly injective because otherwise $\Gamma$ would contain a nontrivial element acting trivially on $\matRP^{n}$. The image of $\Gamma$ is the desired lift $\widehat\Gamma$.
\end{proof}

We also need to show how to promote a cocompact group action on {$\Omega$} to a cocompact action on the cone above $\Omega$.

Let {$\Gamma$ divide $\Omega$}, and let $\widehat{\Gamma}<\SL^\pm_{n+1}(\matR)$ denote the lift of $\Gamma$ as in Proposition \ref{prop:lift}. We now fix a real number $\lambda>1$. The group $\tilde{\Gamma}$ generated by $\widehat{\Gamma}$ and the matrix $\lambda \cdot \mathrm{id}$ is a discrete subgroup of $\mathrm{GL}_{n+1}(\mathbb{R})$ that divides $C$ in the sense of Definition \ref{defn:divide-cone}. We call the group $\tilde{\Gamma}$ the \emph{lift of $\Gamma$ to $\mathrm{GL}_{n+1}(\mathbb{R})$}. Notice that $\tilde{\Gamma}$ depends on the choice of $\lambda$. However, the intersection $\tilde{\Gamma}\cap \mathrm{SL}_{n+1}(\mathbb{R})=\widehat{\Gamma}$ is independent of this choice.

This construction allows us to easily prove that if $\Omega_1$ and $\Omega_2$ are divisible properly convex domains, then their join $\Omega_1 \ast \Omega_2$ is divisible. Let $\Gamma_1<\mathrm{Aut}(\Omega_1)$ and $\Gamma_2<\mathrm{Aut}(\Omega_2)$ be discrete subgroups that divide $\Omega_1=\matP(C_1)\subset\matP(V_1)$ and $\Omega_2=\matP(C_2)\subset\matP(V_2)$, respectively.  Consider their lifts $\widehat{\Gamma}_1$ and $\widehat{\Gamma}_2$ to $\mathrm{SL}^\pm(V_1)$ and $\mathrm{SL}^\pm(V_2)$ respectively. Now, fix two real numbers $\lambda_1,\lambda_2>1$ which are commensurable in a multiplicative sense (meaning that $\lambda_1^m=\lambda_2^n$ for some positive $m,n$ or, equivalently, that there exists $\lambda >1$ such that $\lambda^n=\lambda_1$ and $\lambda^m=\lambda_2$) and consider the lifts $\tilde{\Gamma}_1$ and $\tilde{\Gamma}_2$ to $\mathrm{GL}(V_1)$ and $\mathrm{GL}(V_2)$ defined using $\lambda_1$ and $\lambda_2$, respectively. 

The direct product $\tilde{\Gamma}_1 \times \tilde{\Gamma}_2$ divides $C_1 \oplus C_2 \subset V_1 \oplus V_2$, and intersects the centre of $\mathrm{GL}(V_1\oplus V_2)$ in the group of scalar matrices $\{\lambda^m \cdot \mathrm{id} \, |\,m \in \matZ\}$. The projectivisation $\matP(\tilde{\Gamma}_1 \times \tilde{\Gamma}_2)<\mathrm{PGL}(V_1 \oplus V_2)$ is a discrete group that divides $\Omega_1 \ast \Omega_2$. It is crucial for the discreteness of $\matP(\tilde{\Gamma}_1 \times \tilde{\Gamma}_2)$ that the lifts $\tilde{\Gamma}_1$ and $\tilde{\Gamma}_2$ are defined using real numbers $\lambda_1,\lambda_2>1$ which are commensurable in a multiplicative sense. Indeed we have that $$\mathrm{Aut}(\Omega_1 \ast \Omega_2) \cong \mathrm{Aut}(\Omega_1) \times \Aut(\Omega_2) \times \matP(\matR_{+}^2),$$ and $(\matP(\tilde{\Gamma}_1 \times \tilde{\Gamma}_2)) \cap \matP(\matR_{+}^2)$ is discrete if and only if $(\tilde{\Gamma}_1 \times \tilde{\Gamma}_2) \cap \matR_+ \cdot \mathrm{id}$ is non-trivial.

\subsection{Symmetric domains} \label{sec:symmetric-domains}
A properly convex domain $\Omega$ is said to be \emph{homogeneous} if $\mathrm{Aut}(\Omega)$ acts transitively on $\Omega$.
Given a properly convex cone $C\subset V$, we define its \emph{dual cone} as $C^{\star}=\{f \in V^{\star}\,|\;f(v)>0\, \forall v \in \overline{C} \setminus \{0\}\}\subset V^{\star}$.
A properly convex domain $\Omega=\matP(C)$ is said to be \emph{symmetric} if it is homogeneous and projectively equivalent to $\Omega^{\star}=\matP(C^{\star})$. All symmetric convex cones are divisible. By a result of Vinberg \cite{V2}, if $\Omega$ is homogeneous, then $\Omega$ is symmetric if and only if $\mathrm{Aut}(\Omega)$ is unimodular. Since cocompact lattices only exist in unimodular Lie groups, if $\Omega$ is homogeneous and divisible then it is necessarily symmetric.

The indecomposable symmetric convex domains have been classified by Koecher and Vinberg. Before stating their theorem, let us introduce some notation.

The Lorentz cone
$$\calL_{n+1} = \{(x_0, \ldots, x_n) \in \matR^{n+1}\, |\, -x_0^2+x_1^2+\dots+x_n^2<0,\, x_0>0\} $$
has $\Aut(\calL_{n+1}) \cong \mathrm{O}^+_{n,1}(\matR)\times \matR_{+}$. Its projectivisation $\Omega \subset \matRP^n$ is the projective model of the real hyperbolic $n$-space, with $\Aut(\Omega)^\circ \cong \SO^+_{n,1}(\matR)$. We include with the same notation for $n=0$ also the half-line $\calL_1 = \{x \in \matR \, |\, x>0\}$, with $\Aut(\calL_1) \cong \matR_+$ and $\Aut(\Omega) = \{ \id \}$. 

For $\matK = \matR, \matC, \matH$ and an integer $m \geq 3$, the cone
$$\calP_m(\matK) = \{\mathrm{positive \ definite \ hermitian}\ m \times m\ \mathrm{matrices \ over \ } \matK \}$$
has $\Aut(\calP_m(\matK)) \cong \GL_m(\matK)$, acting by $\mathrm{GL}_{m}(\matK) \times \calP_m(\matK) \ni (A,M) \mapsto AMA^* \in \calP_m(\matK)$, and real dimension
$$n + 1 = \begin{cases} (m^2 + m)/2 & \mbox{if}\ \matK = \matR, \\
m^2 & \mbox{if}\ \matK = \matC, \\
2m^2 - m & \mbox{if}\ \matK = \matH.
\end{cases}$$
Its projectivisation $\Omega \subset \matRP^n$ has $\Aut(\Omega)^\circ \cong \PSL_m(\matK)$. We include with the same notation also the octonionic case $\matK = \matO$ with $m = 3$ and $n+1=27$, where $\calP_3(\matO)$ is a convex cone in the exceptional Jordan algebra $\mathcal{J}$ of $3 \times 3$ hermitian octonionic matrices, and $\Aut(\calP_3(\matO))^\circ \cong E_6^{-26} \times \matR_+$ and $\Aut(\Omega)^\circ \cong E_6^{-26}$. Here $E_6^{-26}$ denotes the exceptional Lie group of determinant-preserving linear transformations of $\mathcal{J}$.

\begin{thm}[Koecher \cite{Koe}, Vinberg \cite{V}]\label{teo:indecomposable-symmetric}
    A symmetric properly convex domain is a join of indecomposable symmetric domains. A symmetric, indecomposable, properly convex domain is projectively equivalent to the projectivisation of one and only one of the above convex cones $\calL_{n+1}$, $\calP_m(\matR)$, $\calP_m(\matC)$, $\calP_m(\matH)$, or $\calP_3(\matO)$. 
\end{thm}

\subsection{Centre and Zariski closure}
We recall here two important results by Vey and Benoist that we will use heavily in the following. We begin by recalling a result of Vey, expanded by Benoist, which characterises the centraliser in $\mathrm{GL}_{n+1}(\matR)$ of a group that divides a convex cone (recall Definition \ref{defn:divide-cone}).

\begin{thm}[Benoist \cite{B2}, Vey \cite{Vey}]\label{teo:Vey} 
    Let $C \subset \matR^{n+1}$ be a properly convex cone, $\widetilde{\Gamma} <\mathrm{Aut}(C)$ a group that spans $C$, and $H$ the centraliser of $\widetilde{\Gamma}$ in $\mathrm{GL}_{n+1}(\matR)$. Then:
    \begin{enumerate}
        \item \label{it:vey:1} There exists a direct sum decomposition $\matR^{n+1}=V_1 \oplus \dots \oplus V_{k}$ of $\matR^{n+1}$ into $\widetilde{\Gamma}$-invariant subspaces such that $$H=\{A \in \mathrm{GL}_{n+1}(\matR) \, |\, \forall i=1,\dots,k\; \exists \lambda_i \in \matR^*:\, Av=\lambda_i v\; \forall v \in V_i\}\cong \mathbb{R}^k.$$
        \item \label{it:vey:2} The cone $C$ decomposes as $C=C_1 \oplus \dots \oplus C_k$, where each $C_i \subset V_i$ is a properly convex cone.
        \item \label{it:vey:3} If $\widetilde{\Gamma}$ is discrete, the action of $\widetilde{\Gamma}$ on each $V_i$ is irreducible, each convex cone $C_i\subset V_i$ is divisible and the centre $Z=\widetilde{\Gamma} \cap H$ of $\widetilde{\Gamma}$ is a lattice in $H$, and is thus isomorphic to $\matZ^k$.
    \end{enumerate}
    If $\widetilde{\Gamma}$ is Zariski connected, then the  cones $C_1, \ldots, C_k$ are indecomposable. Otherwise, $\widetilde{\Gamma}$ has a Zariski-connected subgroup of finite index. 
\end{thm}

We also recall Benoist's Theorem on the Zariski closure of a group that divides a convex cone:
\begin{thm}[Benoist \cite{B2}]\label{teo:Benoist}
Let $C \subset \matR^{n+1}$ be a properly convex cone, and let $\widetilde{\Gamma}<\mathrm{Aut}(C)$ be a discrete, Zariski-connected group that divides $C$. Let $\matR^{n+1}=V_1 \oplus \dots \oplus V_k$ be the direct sum decomposition of $\matR^{n+1}$ into $\widetilde{\Gamma}$-invariant subspaces given in Theorem \ref{teo:Vey} and $C = C_1 \oplus \dots \oplus C_k$ be the corresponding decomposition of $C$ into a direct sum of indecomposable convex cones $C_i \subset V_i$.

The Zariski closure $Z$ of $\widetilde{\Gamma}$ in $\mathrm{GL}_{n+1}(\matR)$ decomposes as 
$$Z=H_1 \times \dots \times H_k,$$ where each $H_i$ is a reductive subgroup of $\mathrm{GL}(V_i)$ with the following properties:
\begin{enumerate}
\item \label{it:benoist:1} If $\Omega_i=\matP(C_i)$ is symmetric, then $H_i$ and $\mathrm{Aut}(C_i)$ are commensurable. In particular $H_i^\circ \cong \mathrm{Aut}(C_i)^{\circ}$ and $\matP(H_i)^{\circ} \cong \mathrm{Aut}(\Omega_i)^{\circ}$ is one of the groups listed in Theorem \ref{teo:indecomposable-symmetric}. 
\item \label{it:benoist:2} If $\Omega_i=\matP(C_i)$ is not symmetric, then $H_i=\mathrm{GL}(V_i)$ and $\matP(H_i)^{\circ}=\mathrm{PSL}(V_i)$.
\end{enumerate}
\end{thm}

\begin{rem}\label{oss:non-simmetrico-discreto}
In particular, it follows from Theorem \ref{teo:Benoist} that if $\Omega \subset \matRP^n$ is a non-symmetric, indecomposable, divisible domain, then $\mathrm{Aut}(\Omega)^{\circ}$ is a Zariski-dense subgroup of the simple Lie group $\mathrm{PGL}_{n+1}(\matR)$. However, a Zariski-dense subgroup of a simple algebraic Lie group is either discrete or dense. The latter case is excluded, because $\mathrm{Aut}(\Omega)^{\circ}$ preserves $\Omega$.  Therefore we necessarily have that $\mathrm{Aut}(\Omega)^{\circ}$ is discrete.
\end{rem}

We record two consequences of Theorems \ref{teo:Vey} and \ref{teo:Benoist} that will be of primary importance in what follows:

\begin{cor} \label{cor:not-product} 
    If $\Gamma$ divides $\Omega$, the following are equivalent:
    \begin{enumerate}
        \item $\Omega$ is decomposable;
        \item $\Gamma$ virtually has infinite centre. 
    \end{enumerate}
\end{cor}

\begin{proof}
    The equivalence directly follows from Benoist and Vey's Theorem \ref{teo:Vey}--\eqref{it:vey:3} applied to the lift $\widetilde{\Gamma}$ of $\Gamma$ to $\mathrm{GL}_{n+1}(\mathbb{R})$ (see Section \ref{sec:lifts}).
\end{proof}
    
\begin{cor}\label{cor:reducible-implies-decomposable}
    If a product $\Gamma \cong \Gamma_1 \times \Gamma_2$ of non-trivial groups divides $\Omega$, then $\Omega$ is decomposable.
\end{cor}

\begin{proof}
    Let us assume, by contradiction, that $\Omega$ is indecomposable. If $\Omega$ is symmetric, by Koecher and Vinberg's Theorem \ref{teo:indecomposable-symmetric}, the group $\Gamma$ is virtually isomorphic to a lattice in $G = \mathrm{SO}^+_{n,1}(\mathbb{R})$, $\mathrm{PSL}_m(\matK)$ with $\matK \in \{ \matR, \matC, \matH \}$, or $E_6^{-26}$, and hence is Zariski dense in $G$ by Borel's density theorem. If instead $\Omega$ is not symmetric, $\Gamma$ is Zariski dense in $G = \PGL_{n+1}(\matR)$ by Benoist's density Theorem \ref{teo:Benoist} (see Remark \ref{oss:non-simmetrico-discreto}). So, in both cases, $\Gamma$ is a Zariski-dense subgroup of a simple (algebraic) Lie group $G$.
    
    For any subgroup $H < G$, let $\overline{H}$ denote its Zariski closure in $G$. Since $\Gamma_i$ is normal in $\Gamma$, its Zariski closure $\overline{\Gamma}_i$ is normal in $\overline{\Gamma} = G$. But $G$ is simple and $\Gamma_i$ is non-trivial, so $\overline{\Gamma}_i = G$. Since moreover $\Gamma_1$ and $\Gamma_2$ commute, $\overline{\Gamma}_1 = G$ and $\overline{\Gamma}_2 = G$ commute. In other words $G$ is abelian, and this is a contradiction.
    \end{proof}

\subsection{Convex projective models}\label{sec:projective-models}

Let $X$ be a smooth $n$-manifold and $G$ a group of analytic
diffeomorphisms of $X$. A \emph{convex projective model} for the pair $(G,X)$ is the datum of a properly convex domain $\Omega \subset \matRP^n$, together with a diffeomorphism $\phi \colon X \rightarrow \Omega$ which carries the elements of $G$ into projective automorphisms of $\Omega$, in the sense that $\phi\circ g \circ \phi^{-1} \in \mathrm{Aut}(\Omega)$ for all $g \in G$.

If $M$ is a manifold with a complete $(G,X)$-structure such that $\Gamma \cong \pi_1(M)$ is the image of the holonomy map and $\phi\colon X \rightarrow \Omega$ is a convex projective model, then $\phi$ descends to a diffeomorphism $$f\colon M\cong X/\Gamma \rightarrow \Omega/\phi \Gamma \phi^{-1}$$ so that $M$ is also naturally endowed with a convex projective structure. 

The symmetric indecomposable domains in Theorem \ref{teo:indecomposable-symmetric} are all convex projective models for pairs of the form $(G,X)$, where $X$ is the symmetric space associated to one of the Lie groups $\mathrm{SO}_{n,1}(\mathbb{R})$, $\mathrm{SL}_m(\matK)$ with $\matK \in \{ \matR, \matC, \matH \}$ or $E_6^{-26}$, and $G=\mathrm{Isom}(X)^{\circ}$. Notice that when $X=\mathbb{H}^n$, then $\Omega$ is the Beltrami--Klein model for hyperbolic space, and $\mathrm{Aut}(\Omega)=\mathrm{PO}_{n,1}(\matR)$ is the full isometry group of $\mathbb{H}^n$. In the other cases we cannot choose for $G$ the full isometry group of $X$. For instance, when $\Omega=\matP(\mathcal{P}_m(\matR))$, the geodesic symmetry $s_0$ at the point corresponding to the identity matrix corresponds to the map $A\rightarrow A^{-1}$, which does not arise via the projectivisation of a linear map of the vector space of real $m \times m$ symmetric matrices.

However, for all connected symmetric spaces, $\mathrm{Isom}(X)$ has finitely many connected components, and therefore for any lattice $\Gamma<\mathrm{Isom}(X)$ the group $\Gamma \cap \mathrm{Isom}(X)^{\circ}$ has finite index in $\Gamma$. We therefore obtain the following:

\begin{prop} \label{prop:projective-model}
    Let $X$ be a symmetric space associated to one of the Lie groups $\mathrm{SO}_{n,1}(\mathbb{R})$, $\mathrm{SL}_m(\matK)$ with $\matK \in \{ \matR, \matC, \matH \}$ or $E_6^{-26}$, and $\Omega$ the indecomposable symmetric domain such that $\mathrm{Aut}(\Omega)^{\circ}$ is isomorphic to $\mathrm{Isom}(X)^{\circ}$. If $\Gamma$ is a cocompact lattice in $\mathrm{Isom}(X)$, then there exists a torsion-free subgroup $\Gamma'<\Gamma$ of finite index such that $M=X/\Gamma'$ is diffeomorphic to the convex projective manifold  $\Omega/\phi\Gamma'\phi^{-1}$.
\end{prop}

Furthermore, one can easily build convex projective models for products of the form $$(G_1 \times G_2 \times \matR, X_1 \times X_2 \times \matR)$$ by considering decomposable domains of the form $\Omega_1 \ast \Omega_2$, where $\Omega_1$ and $\Omega_2$ are convex projective models for $(G_1,X_1)$ and $(G_2,X_2)$, respectively.

\section{NPC locally symmetric spaces} \label{sec:spazi-simmetrici}

In this section we prove Theorem \ref{teo:NPC-locally-symmetric-intro}. We address the following question: {\itshape Can a compact locally symmetric space be diffeomorphic/homeomorphic/homotopically equivalent to a compact convex projective manifold?} A partial, negative, answer is already given for the case of real rank one by Proposition \ref{prop:rank-1}, which will be applied in this section. We will make an essential use of superrigidity: see Section \ref{sec:rigidity} for this and related results, including the proof of a technical lemma that will be applied here. 

We make a preliminary observation: {by asphericity, if a closed locally symmetric space $X/\Lambda$ and a closed convex projective manifold $M$ are homotopically equivalent, then (they have the same dimension and) $X$ cannot have compact factors in its De Rham decomposition. If $n\neq4$, by the topological rigidity $X/\Lambda$ and $M$ are homeomorphic (see Section \ref{sec:prel-0}).}

We will split our analysis into two cases. The first is the case of locally symmetric spaces of non-compact type. Then we will deal with the more general case of non-positively curved locally symmetric spaces.

\subsection{Non-compact type}\label{sec:non-compact-type}

We begin by considering the case where the symmetric space $X$ is \emph{of non-compact type}, there is no compact or Euclidean factor in its De Rham decomposition. We prove:

\begin{prop}\label{prop:localmente-simmetrico} 
Let $X$ be a symmetric space of non-compact type and $\Gamma$ a cocompact lattice in $G =\mathrm{Isom}(X)^{\circ}$. If $\Gamma$ divides a properly convex domain $\Omega \subset \matRP^{n}$, then $\Omega$ is indecomposable, $\Gamma$ is irreducible, and: 
\begin{enumerate}
    \item \label{it:non-cpt-type:hyp} if $\Omega$ is non-symmetric, then $G=\mathrm{SO}^+_{n,1}(\mathbb{R})$;
    \item \label{it:non-cpt-type:non-hyp}  if $\Omega$ is symmetric (and therefore projectively equivalent to one of the domains listed in Theorem \ref{teo:indecomposable-symmetric}), then $G\cong \mathrm{Aut}(\Omega)^{\circ}$.
    \end{enumerate}
\end{prop}

In the statement above and in the rest of the paper, we say that an arbitrary group $\Gamma$ \emph{divides} a properly convex domain $\Omega$ if $\Gamma$ is isomorphic to a group (of projective transformations) which divides $\Omega$. 

It is well known that the situation described in case \eqref{it:non-cpt-type:hyp}.

\begin{rem} 
    In case \eqref{it:non-cpt-type:non-hyp} of Proposition \ref{prop:localmente-simmetrico}, $\Omega$ is a convex projective model for the pair $(G,X)$, so that the locally symmetric space $N=X/\Gamma$ is diffeomorphic to the convex projective manifold $M = \Omega/\Gamma$. In case \eqref{it:non-cpt-type:hyp}, despite $\Omega$ and $X$ being diffeomorphic, {if $n > 3$ we do not know if there is a} $\Gamma$-equivariant diffeomorphism between $X$ and $\Omega$, and so we cannot prove that $N$ is diffeomorphic to $M$. On the other hand, if $n \neq 4$, by topological rigidity (recall Section \ref{sec:prel-0}) $M$ and $N$ are homeomorphic.
\end{rem}

\begin{rem}\label{rem:proj-rigidity-non-compact-type}
    If $\Gamma<G$ is a lattice as in Proposition \ref{prop:localmente-simmetrico} and $G$ has real rank $\geq 2$, then $G$ is not isomorphic to $\mathrm{SO}^+_{n,1}(\matR)$, we are in case \eqref{it:non-cpt-type:non-hyp}, and $\Gamma$ is \emph{projectively rigid}: all representations $\Gamma \rightarrow \mathrm{PGL}_{n+1}(\matR)$ of $\Gamma$ as a group that divides a properly convex domain in $\matRP^n$ are conjugate in $\mathrm{PGL}_{n+1}(\matR)$. This holds true because up to conjugation there is a unique realisation of $\Gamma$ as a lattice in $G$ by the Superrigidity Theorem \ref{teo:superrigidity} and a unique realisation of $G$ as $\mathrm{Aut}(\Omega)^{\circ}$ for some symmetric properly convex domain $\Omega\subset \matRP^{n}$ up to projective equivalence by Theorem \ref{teo:indecomposable-symmetric}.
\end{rem}

We will give separate proofs for the various statements of Proposition \ref{prop:localmente-simmetrico}. We begin by recalling some useful facts. 

A connected Lie group $G$ is \emph{algebraic} if it is isomorphic to the identity component $\mathbf{G}(\mathbb{R})^{\circ}$ of the group of real points of a linear algebraic $\mathbb{R}$-group $\mathbf{G}$. We will make use of the following well-known facts about algebraic Lie groups \cite[\S 6, p.328]{M}:
\begin{enumerate}
    \item the centre of a semisimple algebraic Lie group $G$ is finite;
    \item a connected, semisimple, real Lie group with trivial centre is algebraic.
\end{enumerate}
Notice that if $X$ is a symmetric space of non-compact type, the Lie group $G=\mathrm{Isom}(X)^{\circ}$ is semisimple and has trivial centre \cite[p.334]{M}, and therefore it is algebraic.

We begin by proving that, under the hypotheses of Proposition \ref{prop:localmente-simmetrico}, the domain $\Omega$ is indecomposable. 

\begin{lemma}\label{lem:indecomposable-domain} 
    Let $G$ be a connected semisimple algebraic Lie group of non-compact type, and $\Gamma<G$ a lattice. If $\Gamma$ divides a properly convex domain $\Omega \subset \matRP^{n}$, then $\Omega$ is indecomposable.
\end{lemma}

Lemma \ref{lem:indecomposable-domain} follows immediately from Corollary \ref{cor:not-product} and the following well-known fact:

\begin{lemma}\label{lemma:finitecenter}
    If $G$ is a connected semisimple algebraic Lie group of non-compact type and $\Gamma<G$  is a lattice, then $\Gamma$ has finite centre. If furthermore $G$ has trivial centre, so does $\Gamma$.
    \end{lemma}
\begin{proof}
Denote by $Z$ the centre of $\Gamma$ and let $C_{G}(Z)$ denote the centraliser of $Z$ in $G$. We have that $C_{G}(Z)$ is Zariski closed in $G$, since the property of commuting with an element $g \in G$ is expressible via polynomial equations, and clearly $\Gamma < C_{G}(Z)$. By Borel's Density Theorem \cite{Borel-density} $\Gamma$ is Zariski dense in $G$, and thus $C_{G}(Z)=G$. It follows that $Z$ is contained in the (finite) centre of $G$.

\end{proof}

We now prove the irreducibility of $\Gamma$ in the statement of Proposition \ref{prop:localmente-simmetrico}:

\begin{lemma}\label{lem:irriducibile}
    Under the hypotheses of Proposition \ref{prop:localmente-simmetrico}, the lattice $\Gamma$ is irreducible.
\end{lemma}

\begin{proof} 
    Assume, by contradiction, that $\Gamma$ is reducible. This means that $G=\mathrm{Isom}(X)^{\circ}$ decomposes as a direct product of groups $G_1 \times \dots \times G_{\ell}$ with $\ell\geq 2$ and, up to possibly passing to a finite-index subgroup, we can assume that $\Gamma=\Gamma_1 \times, \dots, \times \Gamma_{\ell}$, where each $\Gamma_i$ is an irreducible lattice in $G_i$. Then $\Omega$ is decomposable by Corollary \ref{cor:reducible-implies-decomposable}, but this contradicts Lemma \ref{lem:indecomposable-domain}
\end{proof}

We conclude the proof of Proposition \ref{prop:localmente-simmetrico}:

\begin{proof} [Proof of Proposition \ref{prop:localmente-simmetrico}]
    We first assume that $\Omega$ is not symmetric. Let us denote by $\tilde{\Gamma}$ the lift of $\Gamma$ to $\mathrm{GL}_{n+1}(\mathbb{R})$. Up to possibly passing to a finite-index subgroup, we may assume that $\tilde{\Gamma}$ is Zariski connected. Then it follows from Lemma \ref{lem:indecomposable-domain} and Benoist's density Theorem \ref{teo:Benoist} that $\tilde{\Gamma}$ is Zariski dense in $\mathrm{GL}_{n+1}(\mathbb{R})$ and $\Gamma$ is Zariski dense in $\mathrm{PGL}_{n+1}(\mathbb{R})$.

    We now reason case by case, using the fact that $\Gamma$ is irreducible by Lemma \ref{lem:irriducibile}. If the Lie group $G$ has rank $\geq 2$, by the Superrigidity Theorem \ref{teo:superrigidity}, we obtain a continuous morphism $\phi\colon G \rightarrow \mathrm{PSL}_{n+1}(\mathbb{R})$. Moreover, since $\phi(\Gamma)$ is discrete in $\mathrm{PSL}_{n+1}(\mathbb{R})$, we have that $\phi$ is an isomorphism. Therefore, $X$ is the symmetric space associated to the Lie group $\mathrm{PSL}_{n+1}(\mathbb{R})$ with $n\geq 2$. However, in this case $\mathrm{dim}(X)=n(n+3)/2 \neq n=\mathrm{dim}(\Omega)$, and we obtain a contradiction.

    If $G$ has real rank one and is isomorphic to $\mathrm{PSp}(n,1), n\geq 2$, or the exceptional adjoint group $F_4^{-20}$, by Theorem \ref{teo:superrigidity} we obtain an isomorphism between $G$ and $\mathrm{PSL}_{n+1}(\mathbb{R})$, which is again a contradiction since $\mathrm{PSp}(n,1)$ and $F_4^{-20}$ have absolute type $C_{n+1}$ and $F_4$ respectively, while $\mathrm{PSL}_{n+1}(\mathbb{R})$ has absolute type $A_n$.

    Only two cases remain to be treated, namely the real hyperbolic case (where $G=\mathrm{SO}^+_{n,1}(\mathbb{R})$ with $n \geq 2$) and the complex hyperbolic case (where $G=\mathrm{PSU}_{m,1}$ and $2m=n \geq 4$). These are precisely the Lie groups for which the superrigidity does not hold. However, by Proposition \ref{prop:rank-1} no lattice in $\mathrm{PSU}_{m,1}$ can divide a properly convex domain $\Omega \subset \matRP^{2m}$, so the latter case is excluded. Therefore $\Gamma$ is a lattice in $\mathrm{SO}^+_{n,1}(\mathbb{R})$.

    Finally, if $\Omega$ is symmetric we obtain that $\Omega$ is one of the domains listed in Theorem \ref{teo:indecomposable-symmetric} (and is not a point). The group $\Gamma$ is realised as a cocompact lattice in $G$ and $\mathrm{Aut}(\Omega)^{\circ}$, which are  both connected and have trivial centre and no compact factors. Then Corollary \ref{cor:weak-Mostow} implies that $G$ is isomorphic to $\mathrm{Aut}(\Omega)^{\circ}$.

    In all these cases the domain $\Omega$ is a convex projective model for the pair $(\mathrm{Isom}(X)^{\circ},X)$, and thus there exists a diffeomorphism from $X/\Gamma$ to $\Omega/\Gamma$.    
\end{proof} \label{sec:proof-teo-loc-simm}

\begin{rem} \label{rem:no-corlette-2}
    Recall that the rank-one cases $\mathrm{PSp}(n,1)$ and $F_4^{-20}$ can also be excluded like $\mathrm{PSU}_{m,1}$ by Proposition \ref{prop:rank-1}, without need of Corlette's superrigidity; see Remark \ref{rem:no-corlette}.
\end{rem}

\subsection{The non-positively curved case}\label{sec:NPC-locally-symmetric}
The general case of non-positively curved (NPC) locally symmetric spaces is more intricate. The main difference with the non-compact type case lies in the fact that if the universal cover of an NPC locally symmetric space $M$ contains a nontrivial Euclidean factor, then $\Lambda=\pi_1(M)$ has, at least virtually, infinite centre. As such there is no obstruction for $\Lambda$ to divide a decomposable domain, and the projection of $\Lambda$ to the semisimple part of $\mathrm{Isom}(Y)$ can be a reducible lattice.
In fact in Section \ref{sec:construction} we will construct examples of NPC locally symmetric spaces whose fundamental group divides a decomposable properly convex domain and with both reducible or irreducible projection to the semisimple part.

We begin with some preliminaries on NPC locally symmetric spaces.
Let $Y$ be an NPC symmetric space. We assume that its de Rham decomposition has a nontrivial Euclidean factor, i.e. $Y=X \times \matE^d$ where $d \geq 1$ and $X$ is a symmetric space of non-compact type. We then have the decomposition of the identity component of the isometry group of $Y$:

$$\mathrm{Isom}(Y)^{\circ} = \mathrm{Isom}(X)^{\circ}\times \mathrm{Isom}(\matE^d)^{\circ} = G \times (\matR^d \rtimes \mathrm{SO}_{d}(\matR)),$$
where $G = \mathrm{Isom}(X)^{\circ}$ is a connected, semisimple Lie group with trivial centre, and $\matR^d$ denotes the group of translations in the Euclidean factor. We denote by $p_1\colon \mathrm{Isom}(Y)^{\circ}\rightarrow G$ and $p_2\colon \mathrm{Isom}(Y)^{\circ}\rightarrow \mathrm{SO}_d(\matR)$ the natural projections.

We summarise some results of Eberlein as follows:
\begin{thm}[Eberlein \cite{EB1, EB2, EB3}] \label{thm:Eberlein} 
    Let $Y$ be an NPC symmetric space with non-trivial Euclidean factor, and $\Lambda<\mathrm{Isom}(Y)^{\circ}$ a torsion-free lattice.
    \begin{enumerate}
        \item \label{it:eberlein:1} The unique maximal abelian normal subgroup of $\Lambda$ is the group $C(\Lambda)=\Lambda \cap \matR^d$, and it is a lattice in $\matR^d$, so that $C(\Lambda) \cong \matZ^d$ and $\matE^d/C(\Lambda)=T^d$ is a $d$-torus.
        \item \label{it:eberlein:2} The group $\Gamma=p_1(\Lambda)$ is a torsion-free lattice in $G$.
        \item \label{it:eberlein:3} Up to possibly passing to a finite-index subgroup, we have that $p_2(\Lambda)=\{1\}$, so that $\Lambda<G \times \matR^d$ and the centre of $\Lambda$ is $C(\Lambda)$. Moreover, there is a homomorphism $\rho\colon \Gamma\rightarrow \matR^d/C(\Lambda)<\mathrm{Isom}(T^d)$ such that $\Lambda$ consists of pairs of the form $(\gamma,\eta)\in G\times \matR^d$ for $\eta \in \rho(\gamma)$. 
        \item \label{it:eberlein:4} Up to passing to a further finite-index subgroup, we may assume that $\rho\colon \Gamma \rightarrow \matR^d/C(\Lambda)$ is trivial on the preimage in $\Gamma$ of the torsion subgroup of $H_1(\Gamma;\mathbb{Z})=\Gamma/[\Gamma,\Gamma]$. When this is the case, the exact sequence $$1\rightarrow C(\Lambda)\rightarrow \Lambda\rightarrow \Gamma \rightarrow 1$$ splits, and $\Lambda$ is isomorphic to $\Gamma\times \matZ^d$.
    \end{enumerate}
\end{thm}

Essentially, Theorem \ref{thm:Eberlein} states that any closed NPC locally symmetric space $M=Y/\Lambda$ is virtually diffeomorphic to a product $N \times T^d$, where $N=X/\Gamma$ is a locally symmetric space of non-compact type and $T^d$ is a torus whose dimension $d$ is equal to the dimension of the Euclidean factor in its De Rham decomposition of $X$. 

However, it is by no means true, in general, that $M$ as above should be isometric to a product of $N$ with a flat torus: this will happen only when the homomorphism $\rho\colon \Lambda\rightarrow T^d$ is trivial. Moreover, if the image of the homomorphism $\rho$ is infinite, no finite cover of $M$ is isometric to a Riemannian product of $N$ with a flat torus.

Item \eqref{it:eberlein:3} of Theorem \ref{thm:Eberlein} implies that if $\Lambda<\mathrm{Isom}(Y)$ is a cocompact lattice, there exists a finite-index subgroup $\Lambda'<\Lambda$ such that $\Lambda'<\mathrm{Isom}(X)^{\circ}\times \mathbb{R}^d$. In particular, in order to determine if some finite cover of an NPC locally symmetric space of the form $M=Y/\Lambda$ admits a convex projective structure, we may assume that $\Lambda<\mathrm{Isom}(X)^{\circ}\times \mathbb{R}^d$.

The aim of this Section is to prove the following result (Theorem \ref{teo:NPC-locally-symmetric-intro}), which extends Proposition \ref{prop:localmente-simmetrico} to the case where a lattice acts on a symmetric space with a nontrivial Euclidean factor:
\begin{thm}\label{teo:NPC-locally-symmetric} Let {$X$} be a symmetric space of non-compact type, $Y=X \times \matE^d$ an NPC symmetric space, and $\Lambda<\mathrm{Isom}(X)^{\circ}\times \matR^d$ a cocompact lattice. Then $\Lambda$ divides a properly convex domain $\Omega \subset \matRP^{n}$ if and only if $X$ is of the form $$X=X_1 \times  \dots \times X_r,$$ where $r \leq d+1$ and each $X_i$ is the irreducible symmetric space associated to one of the simple Lie groups $\mathrm{SO}_{m,1} (\mathbb{R})$, $\mathrm{SL}_m(\matK)$ with $\matK \in \{ \matR, \matC, \matH \}$, or $E_6^{-26}$.
\end{thm}

We can be more precise as to how the domain $\Omega$ is built and how $\Lambda$ acts on it. {Let $$G = \Isom(X)^\circ = G_1 \times \dots \times G_r$$ be} the decomposition of $G=\Isom(X)^\circ$ into simple factors, set $\Gamma=p_1(\Lambda)$ the projection of $\Lambda$ to $G$, and denote by $\Gamma_i$ the projection of $\Gamma$ to $G_i$. 
There is a one-to-one correspondence between the simple factors of the decomposition of $G$ and the indecomposable factors $\Omega_1, \dots,\Omega_r$ of $\Omega$ of dimension $\geq 1$, such that the action of $\Lambda$ on $\Omega$ projects to an action by projective automorphisms of $\Gamma_i$ on $\Omega$.

If a factor $G_i$ is of the form $\mathrm{SL}_m(\matK)$ or $E_6^{-26}$, then $\Omega_i$ is the corresponding symmetric domain from the list in Theorem \ref{teo:indecomposable-symmetric}.
If $G_i$ is isomorphic to $\mathrm{SO}^+_{m,1}(\matR)$ and $\Gamma_i$ is dense in $G_i$, then $\Omega_i=\matP(\mathcal{L}_{m+1}) = \matH^m$ as in Theorem \ref{teo:indecomposable-symmetric}. If $G_i$ is isomorphic to $\mathrm{SO}^+_{m,1}(\matR)$ and $\Gamma_i$ is discrete in $G_i$, then $\Gamma_i$ is an irreducible factor of the lattice $\Gamma$ and, similarly to the case of Theorem \ref{teo:indecomposable-symmetric}, $\Omega$ is either $\matP(\mathcal{L}_{m+1}) = \matH^m$ or it is non-symmetric and divided by $\Gamma_i$. In total, $\Omega$ has $d+1$ indecomposable factors, and the centre $C(\Lambda)$ of $\Lambda$ is realised as a lattice in $\matP(\matR_+^{d+1})\cong \matR^d$. Similarly, $\Gamma$ is realised as a discrete subgroup of $\mathrm{Aut}(\Omega_1)^{\circ}\times \dots \times \mathrm{Aut}(\Omega_r)^{\circ}$ and its action on the product $\Omega_1 \times \dots \times \Omega_r$ of indecomposable factors of positive dimension is cocompact.

\subsection{Constructing convex projective structures}\label{sec:construction} 
We prove here the ``if'' part of Theorem \ref{teo:NPC-locally-symmetric}.

\begin{prop}\label{prop:constructing-projective-structures-NPC}
Any cocompact lattice $\Lambda<\mathrm{Isom}(X)^{\circ}\times \mathbb{R}^d$ that satisfies the hypotheses of Theorem  \ref{teo:NPC-locally-symmetric} divides some properly convex domain $\Omega\subset \mathbb{RP}^n$.
\end{prop}

The proof follows easily from a general construction that we now outline.
Let $X$ be a symmetric space of non-compact type with $G=\mathrm{Isom}(X)^{\circ}$, and let $$G=F_1 \times \dots \times F_r,$$ be the decomposition of $G$ into simple factors. We require that the group $G$ contains an irreducible lattice, and also that there exists a convex projective model for all pairs of the form $(F_j,X_j)$, where $X_j$ is a factor of the De Rham decomposition of $X$. It is a well-known corollary of the Superrigidity Theorem \cite[p.335]{M} that if there exists an irreducible lattice in $G$ then $G$ is \emph{isotypic}, meaning that all the simple factors of the complexification of the Lie algebra of $G$ are isomorphic. Combining this with the discussion in Section \ref{sec:projective-models} we assume that:

\begin{enumerate}
\item \label{it:ass:1} $F_1, \ldots, F_r$ are isogenous to $\mathrm{SO}_{m,1}(\mathbb{R})$, $\mathrm{SL}_m(\matK)$ with $\matK \in \{\matR, \matC, \matH \}$, or $E_6^{-26}$, 
\item \label{it:ass:2} $G$ is isotypic,
\end{enumerate}
and we say that the group $G$ is \emph{projectively admissible}. 

Given that $\mathrm{SO}_{2m,1}(\matR)$ has absolute type $B_m$, $\mathrm{SO}_{2m+1,1}(\matR)$ has absolute type $D_{m+1}$, $\mathrm{SL}_m(\matR)$ and $\mathrm{SL}_m(\matC)$ have absolute type $A_{m-1}$, $\mathrm{SL}_m(\matH)$ has absolute type $A_{2m-1}$ and $E_6^{-26}$ has absolute type $E_6$, we have the following possibilities:

\begin{itemize}
    \item Absolute type $A_1\cong B_1\cong D_1$: the factors $F_j$ are all isomorphic to $SO^+_{2,1}(\matR)\cong \mathrm{PSL}_2(\matR)$ or $SO^+_{3,1}(\matR)\cong \mathrm{PSL}_2(\matC) $.
    \item Absolute type $B_m$, $m\geq 2$: the factors $F_j$ are all isomorphic to $\mathrm{SO}^+_{2m,1}(\matR)$.
    \item Absolute type $D_{m+1}$, $m\geq 2$: the factors $F_j$ are all isomorphic to $\mathrm{SO}^+_{2m+1,1}(\matR)$.
    \item Absolute type $A_{2m-1}$, $m\geq 2$: the factors $F_j$ are all isomorphic to $\mathrm{PSL}_m(\matH)$ or $\mathrm{PSL}_{2m}(\matR)$ or $\mathrm{PSL}_{2m}(\matC)$.
    \item Absolute type $A_{2m}$, $m\geq 1$: the factors $F_j$ are all isomorphic to $\mathrm{PSL}_{2m+1}(\matR)$ or $\mathrm{PSL}_{2m+1}(\matC)$.
    \item Absolute type $E_6$: the factors $F_j$ are all isomorphic to $E_6^{-26}$.
\end{itemize}
In the case of absolute type $A_3\cong D_3$ in the list above, we have $\mathrm{PSL}_2(\matH) \cong \mathrm{SO}^+_{5,1}(\matR)$. 

Now, let us consider a finite collection $G_1,\dots, G_{\ell}$ of projectively admissible groups. Denoting by $S_i$ the symmetric space corresponding to $G_i$, we choose a cocompact irreducible lattice $\Gamma_i < G_i$ for every $i=1,\dots,\ell$, and consider the locally symmetric space
$$M=S_1/\Gamma_1 \times \dots \times S_{\ell}/\Gamma_{\ell}.$$
Among the irreducible factors of $\pi_1(M)=\Gamma=\Gamma_1 \times \dots \times \Gamma_{\ell}$ we look at those which correspond to a lattice in a real hyperbolic space. For each such factor $\Gamma_i$, we choose a properly convex domain $\Omega_i$ divided by $\Gamma_i$. By Proposition \ref{prop:localmente-simmetrico}, the domain $\Omega_i$ is indecomposable and is either bounded by an ellipsoid as in Theorem \ref{teo:indecomposable-symmetric} or it is not symmetric.

The remaining factors of $\Gamma$ correspond to an irreducible lattice in a semisimple Lie group of rank $\geq 2$ with associated symmetric space $S_i$. For each factor $X_{i_j}$, $j=1,\dots,q$, of the De Rham decomposition of $S_i$, we choose the (indecomposable) symmetric domain $\Omega_{i_{j}}$ which is a convex projective model for $X_{i_j}$ as in Theorem \ref{teo:indecomposable-symmetric}. Notice that $\Gamma_i$ is naturally realised as a discrete, cocompact, irreducible subgroup of $\mathrm{Aut}(\Omega_{i_1})\times\dots\times \mathrm{Aut}(\Omega_{i_q})$.

Let $\Delta^h \subset \matRP^h$ denote the interior of an $h$-simplex, that is the join of $h+1$ copies of a single point ${p}=\matP(\mathcal{L}_1)$ (seen as a properly convex domain in $\matRP^0$ as in Theorem \ref{teo:indecomposable-symmetric}). Now we form the join $\Omega=\Omega_1 \ast \dots \ast \Omega_r\ast \Delta^h \subset \matRP^n$ of all the chosen properly convex domains with $\Delta^h \subset \matRP^h$, so that
$$\mathrm{Aut}(\Omega_1)\times \dots \times \mathrm{Aut}(\Omega_r) \times \matP(\matR_+^{r+h+1})< \mathrm{Aut}(\Omega)$$
and $\Gamma$ is naturally realised as a discrete, cocompact subgroup of $\mathrm{Aut}(\Omega_1)\times \dots \times \mathrm{Aut}(\Omega_r)$. Notice how $\Omega$ decomposes as a direct product of $r+h+1=d+1$ indecomposable domains: the $\Omega_i$ for $i=1,\dots,r$ and $h+1$ points $\{p_j\}=\Omega_{r+j}$ for $j=1,\dots,h+1$ corresponding to the vertices of the $h$-simplex. Here we allow the case where $h=-1$, meaning that $\Omega=\Omega_1 \ast \dots \ast \Omega_r$ and $d+1=r$.

Finally, we construct a discrete group $\Lambda<\mathrm{Aut}(\Omega)$ that divides $\Omega$ as follows:
\begin{enumerate}
    \item Denote by $\matR_+^{d+1}$ the group of independent positive homotheties the cones above the indecomposable factors $\Omega_i$, and by $\matP(\matR_+^{d+1})$ its projectivisation. Choose a lattice $L\cong \matZ^d<\matP(\matR_+^{d+1})$.
    \item Choose a homomorphism $\rho\colon \Gamma \rightarrow \matP(\matR_+^{d+1})/L$ and take $\Lambda$ to consists of pairs of the form $(\gamma,\eta)\in \prod_{j=1}^{r}\mathrm{Aut}(\Omega_j)\times \matP(\matR^{d+1}_+)$ for $\eta \in \rho(\gamma)$.
\end{enumerate}

\begin{proof} [Proof of Proposition \ref{prop:constructing-projective-structures-NPC}]
    Let $\Lambda<\mathrm{Isom}(X)^{\circ} \times \mathbb{R}^d$ satisfy the hypotheses of Theorem \ref{teo:NPC-locally-symmetric}. The projection $\Gamma=p_1(\Lambda)$ decomposes as a product $\Gamma_1\times \dots \times \Gamma_{\ell}$ of irreducible factors, where each factor $\Gamma_i$ is a lattice in a projectively admissible group. We carry over the construction outlined above as follows:
    \begin{enumerate}
        \item We realise $\Gamma$ as a discrete cocompact subgroup of $\mathrm{Aut}(\Omega_1) \times \dots \times \mathrm{Aut}(\Omega_r)$, where $\Omega_1, \ldots, \Omega_r$ are indecomposable domains of positive dimension in a one-to-one correspondence with the irreducible factors of the De Rham decomposition of $X$. 
        \item We set $\Omega=\Omega_1 \ast\dots\ast \Omega_r \ast \Delta^h$ with $h = d - r$, and choose a lattice $L\cong \mathbb{Z}^d<\mathbb{P}(\mathbb{R}_+^{d+1})$.
        \item We choose an isomorphism between the Lie groups $\mathbb{R}^d/C(\Lambda)$ and $\mathbb{P}(\mathbb{R}^{d+1}_+)/L$ and use it to push the morphism $\Gamma \rightarrow \mathbb{R}^d/C(\Lambda)$ given in Theorem \ref{thm:Eberlein} to a morphism $\rho\colon\Gamma\rightarrow\mathbb{P}(\mathbb{R}^{d+1}_+)/L$.
    \end{enumerate}
We thus obtain a properly convex domain $\Omega$ divided by $\Lambda$. This completes the proof.
\end{proof}

\begin{rem}\label{rem:NPC-projective-structure}
    Proposition \ref{prop:constructing-projective-structures-NPC} produced two closed manifolds, namely the locally symmetric space $N=Y/\Lambda$, where $\Lambda$ is a cocompact lattice in $\mathrm{Isom}(X)^{\circ} \times \matR^{d}$, and the properly convex manifold $\Omega/\Lambda$, having isomorphic fundamental group by construction. If in the construction above we choose all the indecomposable properly convex domains $\Omega_j$ to be symmetric, we obtain that $\Omega=\Omega_1 \ast \dots \ast \Omega_r\ast \Delta^h$ is a convex projective model for the pair $(\mathrm{Isom}(X)^{\circ} \times \matR^{d},Y)$ with $d=r+h$. As a consequence, in this case  $N$ is actually \emph{diffeomorphic} to $\Omega/\Lambda$ and therefore admits a convex projective structure. 
\end{rem}

\subsection{Obstructions for convex projective structures}
The discussion above proves one of the two implications of Theorem \ref{teo:NPC-locally-symmetric}. We now prove the opposite one, which amounts to showing that if $Y=X \times \matE^d$ is an NPC symmetric space and $\Lambda<\mathrm{Isom}(X)^{\circ} \times \matR^d$ is a cocompact lattice that divides a properly convex domain $\Omega \subset \matRP^n$, then $\Omega/\Lambda$ is virtually one of the examples constructed in Section \ref{sec:construction}.

Up to passing to a finite-index subgroup, we may assume that:
\begin{enumerate}
    \item \label{it:ass2:1} $\Lambda$ is isomorphic to $\Gamma \times \matZ^d$, where $\Gamma=p_1(\Lambda)$ is the projection of $\Lambda$ to $\mathrm{Isom}(X)^{\circ}$ and $\Gamma=\Lambda \cap \mathbb{R}^d$ (by Theorem \ref{thm:Eberlein});
    \item \label{it:ass2:2} $\Lambda$ is Zariski connected as a subgroup of $\mathrm{Aut}(\Omega)$, and therefore is realised as a subgroup of $$\mathrm{Aut}(\Omega_1)\times \dots \times \mathrm{Aut}(\Omega_{k+1})\times \matP(\matR_{+}^{k+1})<\mathrm{Aut}(\Omega),$$ where $\Omega_1,\dots,\Omega_{k+1}$ are the indecomposable domains whose join is $\Omega$ (by Theorem \ref{teo:Vey}). 
\end{enumerate}
We will prove the following result, which compares the two representations of $\Lambda$ as a lattice in $\mathrm{Isom}(Y)$ and in $\Aut(\Omega)$.

\begin{thm}\label{teo:ostruzioni-NPC-localmente-simmetrici}
    If $\Lambda \cong \Gamma \times \mathbb{Z}^d$ divides $\Omega$, then
    \begin{enumerate}
    \item \label{it:1} $k=d$, and the centre $\matZ^d$ of $\Lambda$ is realised as a lattice $L$ in the projective group $\matP(\matR^{d+1}_+)$ of independent positive homotheties in the cones above the indecomposable factors of $\Omega$;
    
    \item \label{it:2} the group $\Gamma\cong \Gamma \times \{0\}<\Lambda$ is isomorphic to its image under the projection map $$\pi\colon\prod_{j=1}^{d+1}\mathrm{Aut}(\Omega_j)\times \mathbb{P}(\mathbb{R}_+^{d+1})\rightarrow\prod_{j=1}^{d+1}\mathrm{Aut}(\Omega_j),$$ where the latter is a discrete group such that the quotient $(\Omega_1\times\dots\times \Omega_{d+1})/\Gamma$ is compact;

    \item \label{it:3} the irreducible factors of the representation of $\Gamma\rightarrow\mathrm{Isom}(X)^{\circ}$ are lattices in projectively admissible groups, and they are in a one-to-one correspondence with the irreducible factors of the representation $\pi\colon\Gamma\rightarrow\prod_{j=1}^{d+1} \mathrm{Aut}(\Omega_j)$;
    
    \item \label{it:4} there is a one-to-one correspondence between the irreducible factors of the De Rham decomposition of $X$ and the indecomposable factors of $\Omega$ which are not points;
    
    \item \label{it:5} if an indecomposable factor $\Omega_j$ is symmetric, then it is a projective model for an irreducible factor of the De Rham decomposition of $X$; the projection of $\Gamma$ to $\mathrm{Aut}(\Omega_j)$ is obtained by first projecting $\Gamma$ to the simple factor $F_j<\mathrm{Isom}(X)^{\circ}$ corresponding to $\Omega_j$, and then composing with the morphism $F_j\rightarrow \mathrm{Aut}(\Omega_j)$;
    
    \item \label{it:6} if $\Omega_j$ is non-symmetric, then it is divided by the projection to $\mathrm{Aut}(\Omega_j)$ of an irreducible factor $\Gamma_i$ of $\Gamma$ that corresponds to a lattice in $\mathrm{SO}_{n_j,1}^+(\mathbb{R})$, where $n_j=\mathrm{dim}(\Omega_j)$;
    
    \item \label{it:7} there is a homomorphism $\rho\colon\Gamma \rightarrow \matP(\matR_+^{d+1})/L$ such that $\Lambda$ is realised in $\mathrm{Aut}(\Omega)$ as the group of pairs of the form $(\gamma,\eta)\in \prod_{j=1}^{d+1}\mathrm{Aut}(\Omega_j)\times \matP(\matR^{d+1}_+)$ for $\eta \in \rho(\gamma)$.
    \end{enumerate}
\end{thm}

\begin{proof}
Item \eqref{it:1} easily follows from Theorem \ref{teo:Vey}. Indeed, consider the lift $\widetilde{\Lambda}$ of $\Lambda$ to $\mathrm{GL}_{n+1}(\mathbb{R})$, where $n=\mathrm{dim}(\Omega)$. The centre of $\widetilde{\Lambda}$ is a lattice in the group $\mathbb{R}_+^{k+1}$ of independent homotheties in the cones above the $k+1$ indecomposable factors of $\Omega$, and is therefore isomorphic to $\mathbb{Z}^{k+1}$. The rank $d$ of the centre of $\Lambda$ is therefore equal to $k$.

We now prove item \eqref{it:2}. Since $\Gamma$ intersects $L=\Lambda \cap \mathbb{P}(\mathbb{R}_+^{d+1})$ trivially by \eqref{it:1}, $\pi$ is injective on $\Gamma$. The fact that the discrete group $\pi(\Gamma)$ acts cocompactly on $\Omega_1\times \dots\times \Omega_{d+1}$ follows directly from the cocompactness of the action of $\Lambda$ on $\Omega$.

Let us prove items \eqref{it:3}, \eqref{it:4}, \eqref{it:5} and \eqref{it:6}. We begin by noticing that the group $\prod_{j=1}^{d+1}\mathrm{Aut}(\Omega_j)$ is linear. Indeed each factor of the form $\mathrm{Aut}(\Omega_j)$ is realised as a subgroup of  $\mathrm{SL}_{n_j+1}(\mathbb{R})$, where $n_j=\mathrm{dim}(\Omega_j)$. 

Let us now consider the representation $\Gamma\rightarrow \mathrm{Isom}(X)^{\circ}$ and write the decomposition 
$$\Gamma=\Gamma_1 \times \dots \times \Gamma_{\ell}$$ of $\Gamma$ as a product of irreducible lattices (which certainly exists up to possibly passing to a finite-index subgroup). We denote by $Z_i$ the Zariski closure of $\pi(\Gamma_i)$ inside $\mathrm{GL}_K(\mathbb{R})$, where $K=\sum_{j=1}^{d+1} (n_j+1)$. Since $\Gamma_i$ and $\Gamma_j$ commute for every $i,j$, the same property holds for the respective Zariski closures $Z_i$ and $Z_j$. Moreover the Zariski closure $Z$ of $\pi(\Gamma)$ is generated by the Zariski closures of its irreducible factors.

By applying Benoist's density Theorem \ref{teo:Benoist} to a lift $\widetilde{\Lambda}$ of $\Lambda$ and projecting to the factor $\prod_{j=1}^{d+1}\mathrm{Aut}(\Omega_j)$, we deduce that the group $Z$ decomposes as a direct product of $\mathbb{R}$-simple factors
$$Z=H_1 \times \dots \times H_{d+1},$$
 in a one-to-one correspondence with the indecomposable factors of $\Omega$. For each indecomposable symmetric factor $\Omega_j$, we obtain a factor $H_j$ whose real points are of the form $\mathrm{PSO}_{m,1}(\matR)$ or $\mathrm{PSL}_m(\matK)$ with $\matK \in \{ \matR, \matC, \matH \}$ or $E_6^{-26}$ if $\mathrm{dim}(\Omega_i)\geq 1$, or the trivial group if $\Omega_i$ is a single point. If $\Omega_j$ is non-symmetric, we have that $H_j=\mathrm{PGL}_{n_j+1}(\matR)$.

Each of the groups $Z_1,\dots,Z_\ell$ is a connected, closed, normal subgroup of $Z$, and is therefore equal to a product of some of its $\mathbb{R}$-simple factors. We claim that if $i\neq j$, then $Z_i$ and $Z_j$ are products of distinct simple factors of $Z$. Indeed $Z_i$ and $Z_j$ commute with each other, so if $H_k<Z_i\cap Z_j$ for some $k=1,\dots,d+1$ then $H_k$ is abelian, which is clearly not the case.

Recall that $\Gamma_i$ is an irreducible factor of $\Gamma$ and is therefore an irreducible lattice in a subgroup $G_i<\mathrm{Isom}(X)^{\circ}$.  On the other hand $\Gamma_i$ is isomorphic to its image $\pi(\Gamma_i)< \mathrm{Aut}(\Omega_{1}) \times \dots \times \mathrm{Aut}(\Omega_{i_{d+1}})$, whose Zariski closure is $Z_i$ by definition, and therefore we see $\Gamma_i$ as a Zariski-dense subgroup of $Z_i$. Moreover $\Gamma_i$ acts cocompactly on the product of indecomposable factors $\Omega_{i_1} \times \dots \times \Omega_{i_q}$ where $Z_i=H_{i_1} \times \dots \times H_{i_q}$.

 In order to prove (3) we show that each $\Gamma_i$ is realised as an \emph{irreducible} lattice in $\mathrm{Aut}(\Omega_{i_1}) \times \dots \times \mathrm{Aut}(\Omega_{i_q})$. As a by-product of the case-by-case analysis that we will carry out, we will prove items (4)-(5)-(6). We have the following possible cases:
\begin{itemize}
\item The group $G_i$ is non-simple, and therefore has real rank $\geq 2$. Since $\Gamma_i$ is an irreducible lattice in $G_i$, by the Superrigidity Theorem $G_i$ is isotypic and isomorphic to $Z_i^{\circ}$. 
The projection of $\Gamma_i$ to each simple factor of $G_i$, and therefore to each simple factor of $Z_i^{\circ}$ is dense. This implies that $\mathrm{Aut}(\Omega_{i_k})$ is non-discrete, and by Remark \ref{oss:non-simmetrico-discreto} each $\Omega_{i_k}$ is symmetric (i.e. one of the domains in Theorem \ref{teo:indecomposable-symmetric}). Therefore the group $G_i$ is projectively admissible, and $\Gamma_i$ is realised as an irreducible lattice in $\mathrm{Aut}(\Omega_{i_1}) \times \dots \times \mathrm{Aut}(\Omega_{i_q})\cong G_i$.

\item The group $G_i$ is a simple Lie group of rank $\geq 2$, or $\mathrm{PSp}(m,1)$ or $F_4^{-20}$. Again by the Superrigidity Theorem, $G_i$ is isomorphic to $Z_i^{\circ}$, which is necessarily simple, i.e. $Z_i$ equals a simple factor of $Z$, and $\Gamma_i$ divides a single indecomposable factor $\Omega_{i_1}$  of $\Omega$. By Proposition \ref{prop:localmente-simmetrico}, $\Omega_{i_1}$ is symmetric and $G_i\cong Z_i^{\circ}\cong\mathrm{PGL}_m(\matK)^{\circ}$ with $m\geq 3$, or $E_6^{-26}$. The group $G_i$ is projectively admissible and $\Gamma_i$ is realised as an (irreducible) lattice in $\mathrm{Aut}(\Omega_{i_1})\cong G_i$.

\item The group $G_i$ is either $\mathrm{SO}^+_{m,1}(\matR)$ or $\mathrm{PSU}_{m,1}$. Here we cannot apply the Superrigidity Theorems to conclude that $G_i$ is isomorphic to $Z_i^{\circ}$. We examine the discrete, cocompact action of $\Gamma_i$ on $\Omega_{i_1}\times \dots\times \Omega_{i_q}$, and prove that $q=1$, so that $\Gamma_i$ is an irreducible lattice in $\mathrm{Aut}(\Omega_{i_1})$. We argue by contradiction and assume that $q>1$. 

At least one of $\Omega_{i_1},\dots,\Omega_{i_q}$ must be non-symmetric; otherwise $\Gamma_i$ would be realised as a lattice in the non-simple group $Z_i^{\circ}$. Let $r$ denote the real rank of $Z_i^{\circ}$. Since $Z_i^{\circ}$ has more than one non-compact factor $r\geq2$, and by \cite[Corollary 2.9]{PR} $\Gamma$ contains a free abelian group of rank $\geq 2$, which is impossible for a cocompact lattice in a rank-one Lie group.

If all the indecomposable factors $\Omega_{i_1}, \dots, \Omega_{i_q}$ are non-symmetric, then $\Gamma_i$ is a finite-index subgroup of the discrete group $\mathrm{Aut}(\Omega_{i_1})\times \dots \times \mathrm{Aut}(\Omega_{i_q})$ and it must be reducible. Again, this implies that $\Gamma$ contains a free abelian group of rank $\geq 2$, which is not possible for a lattice in a rank-one Lie group.

Now we set $I=\{i\in \{i_1,\dots,i_q\} \, |\, \Omega_i \text{\ is\ symmetric}\}$ and $I'=\{i_1,\dots,i_q\}\setminus I$. By what we have just shown, both sets are non-empty. We also set $G'=\prod_{i \in I}\mathrm{Aut}(\Omega_i)$ and $D=\prod_{i \in I'}\mathrm{Aut}(\Omega_i)$, so that $\Gamma_i<G' \times D$ is a cocompact lattice in the product of a semisimple Lie group of non-compact type $G'$ with a discrete group $D$. By Lemma \ref{lemma:reducibility} the lattice $\Gamma_i$ is reducible, which again contradicts the fact that $\Gamma$ is a lattice in a rank-one Lie group.

Since $q=1$, we see that $Z_i$ equals a simple factor $H_{j}$ of $Z$ and that $\Gamma_i$ divides the corresponding indecomposable factor $\Omega_{j}$. By Proposition \ref{prop:localmente-simmetrico} $G_i=\mathrm{SO}^+_{m,1}(\matR)$ is projectively admissible, and either $\Omega_{j}=\matP(\mathcal{L}_{n_j+1})$ with $Z_i=\mathrm{SO}_{n_j,1}(\matR)$, or $\Omega_j$ is non-symmetric and $Z_i=\mathrm{PGL}_{n_j+1}(\matR)$ with $n_j=\mathrm{dim}(\Omega_j)$.
\end{itemize}

Finally, we deal with \eqref{it:7}. The short exact sequence $$1\rightarrow L \rightarrow \Lambda \xrightarrow{\pi} \Gamma \rightarrow 1$$ induced by the projection $\pi$ induces the morphism $\rho\colon\Gamma \rightarrow \mathbb{P}(\mathbb{R}_+^{d+1})/L$ via $\rho(\gamma)=\lambda\cdot L$, where $(\gamma,\lambda) \in \prod_{j=1}^{d+1}\mathrm{Aut}(\Omega_j)\times \mathbb{P}(\mathbb{R}_+^{d+1})$ is any element of $\Lambda$ that projects to $\gamma$.
\end{proof}

\begin{proof} [Proof of Theorem \ref{teo:NPC-locally-symmetric}]
    The proof easily follows from  Theorem \ref{teo:ostruzioni-NPC-localmente-simmetrici}, as the latter implies that the representation of $\Lambda=\Gamma \times \mathbb{Z}^d$ as a lattice in $\mathrm{Aut}(\Omega)$ arises from the construction described in Section \ref{sec:construction}. More specifically, items \eqref{it:2}, \eqref{it:3} and \eqref{it:4} show that each irreducible factor $\Gamma_i$ of $\Gamma$ is a lattice in a projectively admissible group $G_i<G=\mathrm{Isom}(X)^{\circ}$, and acts discretely, cocompactly and irreducibly on a product $\Omega_1 \times \dots \times \Omega_{i_q}$ of indecomposable factors of $\Omega$ of positive dimension. 

    If the factor $\Gamma_i$ is not a real hyperbolic lattice, then $\Omega_{i_j}$ is a projective model for an irreducible factor of the de Rham decomposition of the symmetric space associated to $G_i$ and the action of $\Gamma_i$ on $\Omega_{i_j}$ is the ``obvious'' one by item \eqref{it:5}.

    If the factor $\Gamma_i$ is a real hyperbolic lattice then it acts on a single indecomposable factor $\Omega_i$ of $\Omega$, which can be an ellipsoid as in item \eqref{it:5} or non-symmetric as in item \eqref{it:6}.

    By item \eqref{it:1}, the domain $\Omega$ is the join $\Omega=\Omega_1 \ast \dots \ast \Omega_r \ast \Delta^h$, where $\Omega_1 ,\dots, \Omega_r$ are the positive-dimensional indecomposable domains acted upon by $\Gamma$ and $h = d - r$. Moreover, the centre $\mathbb{Z}^d$ of $\Lambda$ is realised as a lattice $L$ in $\mathbb{P}(\mathbb{R}_+^{d+1})$.

    By item \eqref{it:7} the subgroup $\Lambda<\mathrm{Aut}(\Omega)$ consists of the pairs $(\gamma,\eta)\in \prod_{j=1}^{d+1}\mathrm{Aut}(\Omega_j)\times \matP(\matR^{d+1}_+)$ for $\eta \in \rho(\gamma)$. 
\end{proof}

\section{Four-dimensional geometries} \label{sec:4d}

In this section we prove Theorem \ref{thm:geom-4d}. We will first deal with geometric 4-manifolds, and then with 4-manifolds admitting non-trivial geometric decompositions. The terminology is introduced here.

Let $X$ be a simply connected, complete, homogeneous, Riemannian manifold, and $G = \Isom(X)$. Then $X$ is called a \emph{geometry} (in the sense of Thurston) if there exists a lattice $\Gamma < G$ acting freely on $X$. In other words, $X$ is the Riemannian universal covering space of a finite-volume manifold $M = X/\Gamma$, which is said \emph{geometric} (with geometry $X$). 
A geometry $X$ is \emph{aspherical} if $X$ is contractible (so $M$ is aspherical), and \emph{solvable} if $G$ is solvable (so $\Gamma \cong \pi_1(M)$ is solvable). 

A closed manifold $M$ admits a \emph{geometric decomposition} if it has a (possibly empty or disconnected) two-sided, closed hypersurface $F$ such that the connected components $M_1, \ldots, M_k$ of $M \setminus F$, called the \emph{pieces}, are geometric. In other words, each $M_i$ is diffeomorphic to a complete finite-volume manifold $X_i/\Gamma_i$ with geometry $X_i$. The decomposition is \emph{trivial} if $F$ is empty (so $M$ is geometric).

\subsection{Geometric four-manifolds} \label{sec:geometric}

The four-dimensional geometries have been classified by Filipkiewicz \cite{F}; see Hillman's book \cite{H}. They are 18 plus a countable infinite family: 3 are compact, 3 are non-compact products with spheres, and the remaining ones are aspherical. Among the aspherical ones, 5 plus the ones of the infinite family are solvable. The aspherical, non-solvable geometries are the remaining 7:
$$\matH^4, \quad \matH^2(\matC), \quad \matF^4, \quad \matH^2 \times \matH^2, \quad \matH^2 \times \matE^2, \quad \matH^3 \times \matE^1, \quad \widetilde{\matS\matL} \times \matE^1,$$
where $\widetilde{\matS\matL}$ denotes the universal cover of $\SL_2(\matR)$, and $\matF^4$ the tangent space of $\matH^2$. There is no closed manifold with $\matF^4$ geometry. 
Among these 7, $\matH^4$, $\matH^2(\matC)$, $\matH^2 \times \matH^2$, $\matH^2 \times \matE^2$ and $\matH^3 \times \matE^1$ are {symmetric spaces, and convex projective manifolds admitting} any of these 5 geometries have already been studied in Section \ref{sec:spazi-simmetrici}. (Removing the asphericity assumption, the symmetric geometries are all possible products of $\matH^p$, $\matE^q$, $\mathbb S^r$, plus $\matH^2(\matC)$ and $\matC\matP^2$.)

The following results of Islam and Zimmer \cite[Proposition 5.1--(3) and Theorem 1.4]{IZ} hold in the more general context of convex cocompact actions. 

\begin{thm}[Islam--Zimmer \cite{IZ}] \label{thm:IZ}
    Let $\Gamma$ divide $\Omega \subset \matRP^n$, and $A$ be an infinite virtually solvable subgroup of $\Gamma$. Then:
    \begin{enumerate}
        \item \label{it:abelian} $A$ is virtually abelian;
        \item \label{it:centraliser} the centraliser $C_\Gamma(A)$ divides $\Omega \cap \matP(V)$, for some vector subspace $V$ of $\matR^{n+1}.$
    \end{enumerate}
\end{thm}

Item \eqref{it:centraliser} will be used only in  Section \ref{sec:geometric-dec}. We are ready to prove the ``geometric part'' of Theorem \ref{thm:geom-4d}:

\begin{prop} \label{prop:geometric}
    Let $M$ be a closed, aspherical, geometric $4$-manifold. If $\pi_1(M)$ divides an indecomposable domain $\Omega$, then $M$ is real hyperbolic.
\end{prop}

\begin{proof}
    We have $M = X/\Gamma$ for a cocompact lattice $\Gamma$ in $G = \Isom(X)$, and $X$ is contractible. Hence Filipkiewicz's classification leaves us with the following cases:
\begin{itemize}
    \item If $G$ is solvable, so is $\Gamma \cong \pi_1(M)$. Then $\pi_1(M)$ is virtually abelian by Theorem \ref{thm:IZ}--\eqref{it:abelian}, and this contradicts the fact that $\Omega$ is indecomposable by Corollary \ref{cor:not-product}. 
    \item If $X$ is 
    $\matH^2 \times \matE^2, \matH^3 \times \matE^1$ or $\widetilde{\matS\matL} \times \matE^1$, then the radical $\sqrt{G}$ of the identity component $G$ of $\mathrm{Isom}(X)$ equals the centre of $G$ and is isomorphic to $\mathbb{R}$ or $\mathbb{R}^2$. By a theorem of Wang \cite[Corollary 8.28]{Rag}, $\Gamma \cap \sqrt{G}$ is a lattice in $\sqrt{G}$. Therefore $\Gamma$ has infinite centre, and this contradicts Corollary \ref{cor:not-product}.
    \item If $X = \matH^2 \times \matH^2$, {we have a contradiction to Proposition \ref{prop:localmente-simmetrico}.}  
    \item Otherwise, $X$ is the real or complex hyperbolic space, but the complex case is excluded by Proposition \ref{prop:rank-1}.
\end{itemize}
\end{proof}

\subsection{Geometric decompositions} \label{sec:geometric-dec}

Let $M$ admit a non-trivial geometric decomposition. Each piece $M_i = X_i/\Gamma_i$ is diffeomorphic to the interior of a compact manifold $\overline{M}_i$ with boundary. By a little abuse, we will call $\partial \overline{M}_i$ the \emph{boundary} of $M_i$. We call \emph{cusp group} any subgroup of $\Gamma_i$ induced by the inclusion of a component of $\partial \overline{M}_i$ in $\overline{M}_i$. Note that $\pi_1(M)$ is the fundamental group of a graph of groups with vertex groups the fundamental groups $\Gamma_1, \ldots, \Gamma_k$ of the pieces and edge groups some cusp groups. 

We will make use of the following result of Hillman, here conveniently stated for our purposes:

\begin{thm}[Hillman \cite{H}] \label{thm:hillman} 
    If a closed aspherical $4$-manifold $M$ admits a non-trivial geometric decomposition, then one of the following holds:
    \begin{enumerate}
        \item \label{it:HpxEq} the pieces have geometry $\matH^4$, $\matH^3 \times \matE^1$, $\matH^2 \times \matE^2$ or $\widetilde{\mathbb{SL}} \times \matE^1$; equivalently, the cusp groups are virtually abelian;
        \item \label{it:H2C-F4} the pieces have $\matH^2(\matC)$ or $\mathbb{F}^4$ geometry; equivalently, the cusp groups are nilpotent but not virtually abelian;
        \item \label{it:H2xH2irred} the pieces have irreducible $\matH^2 \times \matH^2$ geometry; equivalently, the cusp groups are solvable but not nilpotent;
        \item \label{it:H2xH2red} the pieces have reducible $\matH^2 \times \matH^2$ geometry, and each piece is virtually the product of a closed and a cusped surface; equivalently, each cusp group is virtually isomorphic to the product of a surface group with $\matZ$.
    \end{enumerate}
\end{thm}

\begin{proof}
    Hillman's \cite[Theorem 7.2]{H} gives the first statement of each item, observing that (the geometry of) every piece must be aspherical. The proof easily follows by analysing the (known: see the references in \cite{H}) topology of the cusps in the aspherical geometries and observing that, topologically, they have to match well in $M$. (In particular, if a piece has one of the geometries in an item, then every piece must have one of those geometries.) So, the second statements essentially follow from the proof of \cite[Theorem 7.2]{H}. With one exception in case \eqref{it:H2xH2red}: if a piece has reducible $\matH^2 \times \matH^2$ geometry, it may also be a virtual product of two cusped surfaces. But in this case (there are exactly two pieces and) $M$ is not aspherical \cite[page 141]{H}. 
\end{proof}

We are ready to conclude the proof of Theorem \ref{thm:geom-4d}.

\begin{prop} \label{prop:geom-dec}
    Let $M$ be a closed aspherical $4$-manifold admitting a non-trivial geometric decomposition. If $\pi_1(M)$ divides an indecomposable domain $\Omega$, then all the pieces are real hyperbolic.
\end{prop}

\begin{proof}
    For each non-compact manifold admitting one of the geometries in Theorem \ref{thm:hillman}, the cusps are well known to be $\pi_1$-injective. Therefore the fundamental groups of the pieces inject in $\pi_1(M) = \Gamma$, and the cusp groups in $\Gamma$ correspond to the fundamental groups of the boundary components of the pieces.
    
    Suppose we are in case \eqref{it:HpxEq} of Theorem \ref{thm:hillman}, and assume that a piece $M_i$ is not hyperbolic. Then the piece has geometry $X=\mathbb{H}^3 \times \mathbb{E}^1$, $\mathbb{H}^2 \times \mathbb{E}^2$ or $\widetilde{\mathbb{SL}} \times \matE^1${. Let $G$ be the identity component of $\mathrm{Isom}(X)$ and $\sqrt{G}$ its radical. If $X = \matH^3 \times \matE^1$ set $\matK = \matC$ and $q = 3$, otherwise set $\matK = \matR$ and $q=2$, so that $\PSL_2(\matK)$ is the identity component of $\Isom(\matH^q)$. We have a short exact sequence $$1\rightarrow \sqrt{G} \rightarrow G \xrightarrow{p} \mathrm{PSL}_2(\mathbb{K})\rightarrow 1.$$} In all cases $\sqrt{G}$ equals the centre of $G$, and what distinguishes the case $X=\mathbb{H}^2\times \mathbb{R}^2$ from $X=\widetilde{\mathbb{SL}} \times \matE^1$ is that in the latter case the sequence does not split. 
    
    Let $\Gamma_i$ denote the fundamental group of the piece $M_i$. By \cite[Corollary 8.28]{Rag}  $\Gamma_i \cap \sqrt{G}$ is a lattice and $\Gamma_i$ surjects onto a lattice in $\mathrm{PSL}_2(\mathbb{K})$. There exists a finite-index subgroup $\Gamma'_i$ such that  $p(\Gamma'_i)$ is torsion-free, and $M'_i=X/\Gamma'_i$ is a circle bundle over a cusped hyperbolic $3$-manifold if $X=\mathbb{H}^3 \times \mathbb{E}^1$, or a torus bundle over a cusped hyperbolic surface in the other cases. In all cases the induced short exact sequence
    $$1\rightarrow \Gamma'_i \cap \sqrt{G} \rightarrow \Gamma'_i \xrightarrow{p} p(\Gamma'_i)\rightarrow 1 $$ splits: if $X=\mathbb{H}^3 \times \mathbb{E}^1$, $\mathbb{H}^2 \times \mathbb{E}^2$ this holds true because $G$ is a direct product, while if $X=\widetilde{\mathbb{SL}} \times \matE^1$ one uses the fact that $p(\Gamma'_i)$ is a free group, and a splitting homomorphism $\rho \colon p(\Gamma_i')\rightarrow \Gamma'_i$ can be (freely) specified by the choice an element in the preimage of each standard generator. In all cases, the left factor $\Gamma'_i \cap \sqrt{G}$ is contained in the centre, therefore it commutes with the image of the splitting and $\Gamma'_i$ decomposes as a direct product $\Gamma'_i= (\Gamma'_i\cap \sqrt{G}) \times p(\Gamma'_i)$.

    Let now $g \in \Gamma$ correspond to a non-trivial element in $\Gamma'_i\cap \sqrt{G}$. Clearly $\Gamma_i'$ is a subgroup of the centraliser $C_{\Gamma}(g)$ of $g$ in $\Gamma$. Since $\Gamma_i'$ surjects onto a lattice in $\PSL_2(\matK)$, the group $C_\Gamma(g)$ is not virtually abelian. Indeed, by Tits' alternative, a lattice in $\PSL_2(\matK)$ always contains a non-abelian free subgroup, and a virtually abelian group has no subgroup that surjects onto a non-abelian free group. 
 
    On the other hand, if $P$ is a maximal parabolic subgroup of $p(\Gamma_i')$ (i.e.\ a cusp subgroup of $p(\Gamma'_i)$), the group $H_P<\Gamma'_i$ generated by $P$ and $\Gamma'_i \cap \sqrt{G}$ is a cusp subgroup of $\Gamma'_i < C_{\Gamma}(g)$ isomorphic to $\mathbb{Z}^3$. We plan to obtain a contradiction by showing that $H_P$ has finite index in $C_{\Gamma}(g)$.
 
    By Theorem \ref{thm:IZ}--\eqref{it:centraliser}, $C_\Gamma(g)$ divides a slice $\Omega \cap \matP(V)$ of $\Omega$. In particular, $C_{\Gamma}(g)$ is the fundamental group of a closed aspherical manifold of dimension $d=\mathrm{dim}(V)-1$. We claim that $d=3$. Indeed,  $d \geq 3$ because $C_\Gamma(g)$ contains a cusp subgroup $H_P <\Gamma'_i$, which is isomorphic to $\matZ^3$, and so has cohomological dimension $d \geq 3$. (Recall that the cohomological dimension of a subgroup does not exceed the one of the group.) Moreover $d \neq 4$, otherwise $C_{\Gamma}(g)$ would divide $\Omega$ and its centre would contain $\langle g \rangle \cong \mathbb{Z}$, {contradicting Corollary \ref{cor:not-product}.}

    Since $d=3$, the Poincar\'e duality groups $H_P\cong \mathbb{Z}^3$ and $C_{\Gamma}(g)$ have the same cohomological dimension. Therefore $H_P$ has finite index in $C_{\Gamma}(g)$, and this contradicts the fact that $C_\Gamma(g)$ is not virtually abelian.
    
    We have shown that in case \eqref{it:HpxEq} all the pieces are real hyperbolic. We have now to exclude the other cases of Theorem \ref{thm:hillman}. Cases \eqref{it:H2C-F4} and \eqref{it:H2xH2irred} contradict Theorem \ref{thm:IZ}--\eqref{it:abelian} because the cusp groups are solvable but not virtually abelian, so let us assume to be in case \eqref{it:H2xH2red}.
    
    Each $\Gamma_i$ has a finite-index subgroup $\Gamma'_i$ isomorphic to $\pi_1(F_i) \times \pi_1(B_i)$, for a cusped hyperbolic surface $F_i$ and a closed hyperbolic surface $B_i$. For every non-trivial $g \in \Gamma$ in the free factor $\pi_1(F_i)$ of $\Gamma'_i$, the centraliser $C_\Gamma(g)$ divides a slice $\Omega \cap \matP(V_g)$ by Theorem \ref{thm:IZ}--\eqref{it:centraliser}. As before, the slice has dimension $d=3$: for $d \leq 3$ by Corollary \ref{cor:not-product} because $\Omega$ is indecomposable, and $d \geq 3$ because $C_\Gamma(g)$ has a subgroup of type $\matZ \times \pi_1(B_i)$ (where the $\matZ$ factor is generated by $g$), which has cohomological dimension 3. In particular, $C_\Gamma(g)$ is virtually of type $\matZ \times \pi_1(B_i)$. So, by Benoist and Vey's Theorem \ref{teo:Vey}, the slice $\Omega \cap \matP(V_g)$ is a cone $\{ p \} \ast \Omega_g$, where the two-dimensional slice $\Omega_g = \Omega \cap \matP(W_g)$ is divided by $\pi_1(B_i)$. Moreover, the lifted linear action of $\pi_1(B_i)$ on $W_g$ is irreducible and the one of $\langle g \rangle$ is by homothety.
        
    Now we claim that $W_g = W_{g'}$ for all $g, g' \in \pi_1(F_i) < \Gamma$. Indeed, $\pi_1(B_i)$ preserves both $W_g$ and $W_{g'}$, so it preserves $W_g \cap W_{g'}$. Since moreover the action is irreducible, either $W_g \cap W_{g'}$ is trivial or $W_g = W_{g'}$, but the former case is easily excluded by Grassmann's formula . So the 3-dimensional vector space $W := W_g$ does not depend on $g$ and is preserved by the action of $\pi_1(B_i)$.
        
    As moreover every $g$ in $\pi_1(F_i)$ acts on $W$ by homothety, $W$ is invariant by the action of the larger group $\pi_1(F_i \times B_i) \cong \Gamma'_i$. Since the cusp groups are suitably identified in $\Gamma$, they act in the same way, and so $W$ is preserved by the fundamental groups of all the pieces. Therefore $W$ is preserved by the lifted action of the whole $\Gamma$. In other words, the linear action of $\Gamma$ on $\matR^5$ is reducible, and this contradicts Benoist and Vey's Theorem \ref{teo:Vey} because $\Omega$ is indecomposable.
\end{proof}

The proof of Theorem \ref{thm:geom-4d} is complete, as it follows from Propositions \ref{prop:geometric} and \ref{prop:geom-dec}.

\begin{rem} \label{rem:low-dim-geom-dec} 
    The same proof of Theorem \ref{thm:geom-4d} (but of course less complicated) essentially holds also in lower dimensions. Note indeed that the aspherical, non-solvable, three-dimensional geometries are just $\matH^3$, $\matH^2 \times \matE^1$ and $\widetilde{\matS\matL}$ (while in two dimensions there is only $\matH^2$). A 3-manifold $M$ with one of the two latter geometries is virtually diffeomorphic to the total space of a circle bundle, so $\pi_1(M)$ virtually has infinite centre, and Theorem \ref{thm:IZ}--\eqref{it:abelian} and Corollary \ref{cor:not-product} apply again in the geometric case. For the case of a non-trivial geometric decomposition, since a cusped manifold with geometry $\matH^2 \times \matE^1$ or $\widetilde{\matS\matL}$ is virtually diffeomorphic to the product of a cusped hyperbolic surface with the circle, our arguments with Theorem \ref{thm:IZ}--\eqref{it:centraliser} apply as well. 
\end{rem}

\section{The construction} \label{sec:chi}

In this section we prove Theorem \ref{thm:construction}. It suffices to show the following:

\begin{thm} \label{thgm:onepiece}
    There exists a closed, orientable, convex projective $4$-manifold $M$ with $\chi(M) = 2$, admitting a non-trivial geometric decomposition with 
    a single piece $M_1$.
\end{thm}

Indeed, the flat two-sided hypersurface $F \subset M$ of the decomposition is non-separating, so (any union of connected components of) $F$ gives cyclic covers of $M$ of arbitrary degree $m > 0$, and so convex projective manifolds with $\chi = 2m$, as desired. Each of them admits a geometric decomposition with $m$ pieces, all diffeomorphic to the cusped hyperbolic manifold $M_1$. 

The convex projective manifold $M$ will be built by finding a torsion-free subgroup of a projective reflection group that divides some domain. In other words, $M$ will cover a convex projective Coxeter orbifold.

Let $C \subset \matR^{n+1}$ be a closed, convex, polyhedral cone bounded by finitely-many linear hyperplanes, and $P \subset \matRP^n$ the projectivisation of $C$. We call \emph{(projective) reflection group} a discrete subgroup $\Gamma < \PGL_{n+1}(\matR)$ generated by reflections along the bounding hyperplanes of $P$, such that $P$ is a fundamental domain for the action on the union $\Omega$ of $\Gamma$-translates of $P$. By \emph{(projective) reflection}, we mean the projectivisation of a linear reflection along a co\-di\-men\-sion-one hyperplane. Abstractly, $\Gamma$ is a Coxeter group. If $\Omega$ is properly convex and $\Gamma$ divides $\Omega$, the convex projective orbifold $\Omega/\Gamma$ is stratum-preserving homeomorphic to $P$. We refer to \cite{CLM1,V3} for further details on Coxeter groups, linear or projective reflection groups, and convex projective Coxeter orbifolds.

Essentially, we will just need a couple of well-known facts from the theory of Coxeter groups: finite Coxeter groups are classified, and the finite subgroups of a Coxeter group are precisely the conjugates of the finite Coxeter groups that one reads off from the diagram thanks to the classification. Our convention for Coxeter diagrams is the standard one. Although the orbifold theory is, arguably, not strictly needed, we will adopt its language, as it is natural for our geometric arguments.

\subsection{A prism and a simplex}\label{subsec:prism}

\begin{figure}
    \centering
    \includegraphics{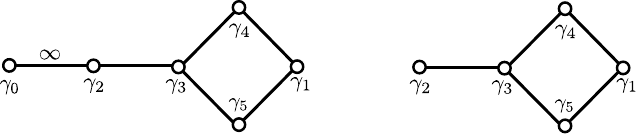}
    \caption{\footnotesize{On the left (resp. right), the Coxeter diagram $D_P$ (resp. $D_S$) of the Coxeter group $\Gamma_P$ (resp. $\Gamma_S$) of the prism $P = \Omega/\Gamma_P$ (resp. the simplex $S = \matH^4/\Gamma_S$). 
    }}
   \label{fig:small-simplex-truncated}
\end{figure}

Let us consider a four-dimensional simplex with facets $F_1, \ldots, F_5$ and vertices $v_1, \ldots, v_5$, where $v_i$ is opposite to $F_i$. Let $S$ and $P$ be obtained from the simplex as follows: $S$ by removing $v_2$, and $P$ by truncating $v_2$ with a new facet $F_0$.

Let now $\Gamma_P$ be the Coxeter group with diagram $D_P$
in Figure \ref{fig:small-simplex-truncated}--left, and $\Gamma_S < \Gamma_P$ the one with diagram $D_S$ in Figure \ref{fig:small-simplex-truncated}--right. As the reader may easily check via the classification of finite Coxeter groups, the nerve posets $\{ \{ i_1, \ldots, i_k \} : |\langle \gamma_{i_1}, \ldots, \gamma_{i_k} \rangle| \neq \infty \}$ of $\Gamma_S$ and $\Gamma_P$, ordered by $\supset$, are isomorphic to the face posets of $S$ and $P$, respectively. Therefore, $S$ and $P$ inherit two orbifold structures as follows. 

Set $G_F = \langle \gamma_{i_1}, \ldots, \gamma_{i_k} \rangle$ for each proper face $F = F_{i_1} \cap \ldots \cap F_{i_k}$, and $G_P = G_S = \{ 1 \}$.  These finite groups act  in the standard way on $\matR^4$ as reflection groups in $\Or(4)$, and the two orbifolds are locally modelled on $\matR^4/G_F$, where $F$ varies among the faces of the polytope. The orbifold fundamental groups of $S$ and $P$ are isomorphic to $\Gamma_S$ and $\Gamma_P$, respectively. Their orbifold Euler characteristic is
$$\chi(P) = 
\chi(S) = \sum_F \frac{(-1)}{|G_F|}^{\dim(F)} = \frac{1}{960},$$
where $F$ runs over the faces of the polytope ($P$ or $S$), including the polytope itself.

It is well known that $S$ is realised as a finite-volume hyperbolic Coxeter simplex $S = \matH^4/\Gamma_S$, where $\Gamma_S$ acts by isometry as a reflection group. Its unique ideal vertex, $v_2$, corresponds to the unique maximal affine subdiagram of $D_S$, spanned by $\gamma_1, \gamma_3, \gamma_4, \gamma_5$. The compact orbifold $P$ admits a convex projective structure as well \cite[Theorem C]{CLM1}:

\begin{thm}[Choi--Lee--Marquis \cite{CLM1}] \label{thm:CLM}
    The Coxeter group $\Gamma_P$ embeds in $\PGL_5(\matR)$ as a reflection group that divides a properly convex domain $\Omega \subset \matRP^4$ so that $P \cong \Omega/\Gamma_P$.
\end{thm}

The truncation facet $F_0$ of $P$ is ``orthogonal'' to the adjacent facets, in the sense that the projective reflection associated to it commutes with the ones of the adjacent facets $F_1$, $F_3$, $F_4$ and $F_5$. 

The convex projective orbifold $P = \Omega/\Gamma_P$ admits an orbifold geometric decomposition in the following sense. The complement $P \sm F_0$ of the truncation facet admits a complete, finite-volume, hyperbolic orbifold structure: the one of $S = \matH^4/\Gamma_S$. The deformation space of convex projective structures on $P$ is a line, and every structure is a deformation of the hyperbolic structure of $S$ \cite[Theorems A and 3.1]{CLM1}.

\subsection{Summary of the proof}

In order to prove Theorem \ref{thgm:onepiece}, it thus suffices to prove the following algebraic fact:

\begin{prop} \label{prop:subgroup}
    There exists a torsion-free subgroup $\Gamma < \Gamma_P$ of index $1920$, whose elements have even word length in the generators $\gamma_0, \ldots, \gamma_5$, and such that the subgroup $\Gamma_1 = \Gamma_S \cap \Gamma$ has index $1920$ in $\Gamma_S$.
\end{prop}

Indeed, the closed convex projective manifold $M = \Omega/\Gamma$ has $\chi(M) = 1920 \cdot \chi(P) = 2$. It is orientable, because $\Gamma_P$ acts on $\Omega$ as a reflection group, reflections in $\Aut(\Omega)$ are orientation reversing, and the elements of $\Gamma \cong \pi_1(M)$ have even word length. Moreover, we can pull back the orbifold geometric decomposition of $P$ to a geometric decomposition of $M$ via the induced orbifold covering $p \colon M \to P$. Indeed, as $P \sm F_0 \cong S$, each component of the complement $M \sm F$ of $F = p^{-1}(F_0)$ is diffeomorphic to the cusped hyperbolic manifold $M_1 = \matH^4/\Gamma_1$ that covers $S = \matH^4/\Gamma_S$. But the orbifold coverings $p$ and $p|_{M \sm F}$ have the same degree, 1920, so $M \sm F \cong M_1$ is connected and Theorem \ref{thgm:onepiece} follows. 

We found it difficult to prove Proposition \ref{prop:subgroup} only via computational methods. After weeks of computation with opportune computer software, we were not even able to know whether or not there exists a torsion-free subgroup of $\Gamma_P$ of minimal index 960 (and so a non-orientable manifold with $\chi = 1$). Note that the complexity of the problem increases with the index. We will rather proceed via a mixture of algebraic and geometric arguments.

The presence of the hyperbolic simplex $S$, a priori not needed for Theorem \ref{thm:construction}, is essential for our proof. Indeed, we first find the subgroup $\Gamma_1 < \Gamma_S$ as the kernel of a quite natural homomorphism of $\Gamma_S$ onto the order-1920 Coxeter group of type $D_5$ (see Figure \ref{fig:D5_diagram}). In particular, the latter group acts on $M_1 = \matH^4/\Gamma_1$ by isometry, with quotient map $M_1 \to S$ a regular orbifold cover. The desired properties of $\Gamma_1$ are conveniently shown by means of the standard representation of the $D_5$-group as a linear reflection group, recalled in Section \ref{sec:weyl}.

Then, to find $\Gamma \supset \Gamma_1$, we exploit the geometry of the construction as follows. By Theorem \ref{thm:CLM}, $\Gamma_1$ (seen as a subgroup of $\Gamma_P$) acts on $\Omega$ as the fundamental group of a projective manifold (tessellated by copies of $P$) with flat geodesic boundary (tessellated by copies of $F_0$), corresponding to the cusps of $M_1$. Since the cover $M_1 \to S$ is regular, the boundary components are pairwise isomorphic. We thus look at the cusps of $M_1$, and find a way to glue the corresponding boundary components in pairs with orientation-reversing maps. In general, gluing with projective isomorphisms is not enough to obtain a genuine convex projective structure on the closed manifold. Our gluings are admissible in the latter sense because we use automorphisms of the covering as gluing maps.

\subsection{The Weyl group of the $D_5$ root system} \label{sec:weyl}
We find convenient to represent the finite Coxeter group of type $D_5$ as a linear reflection group $W(D_5) < \Or(5)$ in the standard way, which we recall here.

We equip $\mathbb{R}^n$ with the standard scalar product. Given a root system $\calR \subset \mathbb{R}^n$, we denote by $W(\calR)$ its Weyl group, i.e. 
$$W(\calR)=\langle R_v \, | \, v \in \calR \rangle,$$
where $R_v$ denotes the reflection in the orthogonal complement to a vector $v \in \calR$. 

Denote by $e_1, \ldots, e_5$ the standard basis of $\mathbb{R}^5$. The \emph{$D_5$ root system} is the set $D_5$ consisting of the $40$ vectors of $\mathbb{R}^5$ of the form
$
\pm e_i \pm e_j
$. A set $\Delta=\{\alpha_1, \alpha_2,\alpha_3,\alpha_4,\alpha_5\}$ of simple roots is obtained by choosing
\begin{equation}\label{eq:simpleroots}
    \alpha_1=e_1-e_2,\quad \alpha_2=e_2-e_3,\quad \alpha_3=e_3-e_4,\quad \alpha_4=e_4-e_5,\quad \alpha_5=e_4+e_5. 
\end{equation}

The \emph{Weyl group $W(D_5)$} is the group generated by the reflections $R_{\alpha_1}, \ldots, R_{\alpha_5}$. Abstractly, it is the finite Coxeter group with diagram in Figure \ref{fig:D5_diagram}. 

\begin{figure}[h]
    \centering
    \includegraphics[width=5.4cm]{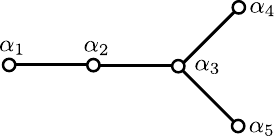}
    \caption{\footnotesize{The Coxeter diagram of $W(D_5)$, with the vertices labelled by the corresponding simple roots.}}
   \label{fig:D5_diagram}
\end{figure}

We now explicitly describe the matrices in this group. Given $v = e_i - e_j \in D_5$, the reflection $R_v$ permutes $e_i$ and $e_j$ while fixing the other vectors of the standard basis. The group generated by reflections of this form is isomorphic to the symmetric group $\mathfrak{S}_5$, represented by permutation matrices. Similarly, given $w = e_i+e_j$, the reflection $R_w$ permutes $e_i$ and $-e_j$ while fixing the other vectors of the standard basis. Then one easily checks that the composition $R_v R_w$ maps $e_i$ and $e_j$ to their opposites while fixing the other vectors of the standard basis. The group generated by these order-two rotations is isomorphic to $(\mathbb{Z}/2\mathbb{Z})^4$, and is represented by unit-determinant diagonal matrices with $\pm 1$ entries. Moreover, we obtain a semidirect product decomposition
$$W(D_5)\cong (\mathbb{Z}/2\mathbb{Z})^4 \rtimes \mathfrak{S}_5,$$
where the action of $\mathfrak{S}_5$ on $(\mathbb{Z}/2\mathbb{Z})^4$ induces the permutation of the diagonal entries. The order of $W(D_5)$ is then the product $2^4 \cdot 5!=1920$ of the orders of the two factors.

\subsection{The cusped hyperbolic manifold} 
We begin by constructing the subgroup $\Gamma_1 < \Gamma_S$, providing a cusped hyperbolic manifold $M_1=\matH^4/\Gamma_1$ as explained in Section \ref{subsec:prism}. The subgroup $\Gamma_1$ will be constructed as the kernel of a homomorphism $\phi$ from the Coxeter group $\Gamma_S$ onto the Weyl group $W(D_5)$. 
{For this purpose, recall that $\Gamma_S$ is generated by $\gamma_1,\ldots,\gamma_5$, and set $$\beta := e_1 - e_5 \in D_5.$$} We will prove the following statement:

\begin{prop}\label{prop:phi and its kernel}
    The assignments
    $$\phi(\gamma_1) = R_\beta, \quad \phi(\gamma_2)=R_{\alpha_2}, \quad \phi(\gamma_3)=R_{\alpha_3}, \quad \phi(\gamma_4)=R_{\alpha_4}, \quad \phi(\gamma_5)=R_{\alpha_5}$$
    define a group homomorphism $\phi\colon \Gamma_S \rightarrow W(D_5)$ such that $\Gamma_1:=\ker(\phi)$ is a torsion-free subgroup of index $1920$ whose elements have even word length in the generators $\gamma_1,\ldots,\gamma_5$.
\end{prop}
\begin{proof}
    We first check that the assignments above specify a group homomorphism, i.e. that all relations in the presentation of $\Gamma_S$ are mapped to the identity element in $W(D_5)$. Notice that, for any two linearly independent vectors $v,w \in D_5$, the order of $R_v R_w$ is determined by their scalar product: if $\langle v,w \rangle=0$ then $R_v$ and $R_w$ commute, while if $\langle v,w \rangle=\pm1$ then $R_v  R_w$ has order $3$. Since $\beta$ is orthogonal to $\alpha_2$ and $\alpha_3$, we see that $\phi(\gamma_1)$ commutes with $\phi(\gamma_2)$ and $\phi(\gamma_3)$. Moreover, since $\langle \beta,\alpha_4\rangle=1$ and $\langle\beta,\alpha_5\rangle=-1$, both $\phi(\gamma_1) \phi(\gamma_4)$ and $\phi(\gamma_1) \phi(\gamma_5)$ have order $3$. The fact that the other relations are satisfied follows directly from the fact that the the images of $\gamma_2,\dots, \gamma_5$ correspond to simple reflections in a basis of a $D_4$ root subsystem.

    We now claim that $\phi$ is surjective. It is sufficient to prove that $R_{\alpha_1}$ lies in the image of $\phi$, since the other generators of $W(D_5)$ are precisely the images of $\gamma_2, \ldots, \gamma_5$. For this we notice that $R_{\beta}  R_{\alpha_4}  R_{\alpha_3}  \alpha_2=-\alpha_1$. Thus, by setting $A = R_{\beta}  R_{\alpha_4}  R_{\alpha_3}$, we see that $R_{\alpha_1} = A R_{\alpha_2} A^{-1}$ belongs to the image of $\phi$. This also shows that {the normal subgroup $\Gamma_1 := \ker(\phi)$ of $\Gamma_S$ has index 1920: the order of $W(D_5)$.} 

    It is now easy to prove that all elements of $\Gamma_1$ are expressed as words of even length in the generators $\gamma_1, \ldots, \gamma_5$. The matrices $\phi(\gamma_1), \ldots, \phi(\gamma_5)$ represent indeed linear reflections along codimension-one hyperplanes, and so have determinant $-1$. Thus, an element $\gamma \in \Gamma_S$ has even word length if and only if $\phi(\gamma)\in W(D_5)$ has unit determinant.

    It only remains to prove that $\Gamma_1$ is torsion free. For this, it is sufficient to check that all the maximal finite subgroups of $\Gamma_S$ are mapped injectively into $D_5$. There are exactly $4$ conjugacy classes of maximal finite subgroups in $\Gamma_S$, one for each finite vertex of $S$. A list of representatives is given below, together with their generators, their realisation as the Weyl group of some root system, and their order. 

    Let $G_i$ denote the group associated to the vertex $v_i$ of $S$ opposite to the facet $F_i$. It is obtained by considering the subdiagram of $D_S$ spanned by the vertices different from $\gamma_i$. Notice that $G_2$ is the affine (and thus infinite) Coxeter group corresponding to the ideal vertex $v_2$ of $S$, so it does not belong to this list. 
\begin{align*} 
    G_1 & \cong W(D_4) \cong (\mathbb{Z}/2\mathbb{Z})^3 \times \mathfrak{S}_4, & | G_1 | = 192;\\
    G_3 &\cong W(A_1) \times W(A_3) \cong \mathbb{Z}/2\mathbb{Z} \times \mathfrak{S}_4, & | G_3 | = ~\, 48;\\
    G_4 & \cong W(A_4) \cong \mathfrak{S}_5, & | G_4 | = 120;\\
    G_5 & \cong W(A_4) \cong \mathfrak{S}_5, & | G_5 | = 120.
\end{align*}

    We are thus left with the task of showing that $\phi$ is injective on $G_i$ for each $i=1,3,4,5$. We begin by dealing with $G_1$. Notice how the set of vectors of $D_5$ whose first coordinate is equal to $0$ forms a copy of the $D_4$ root system in the vector space {$\{ x_1 = 0 \} \subset \matR^5$}. A set of simple roots for this root system is given by $\Delta'=\{\alpha_2, \alpha_3, \alpha_4, \alpha_5\}$ as in \eqref{eq:simpleroots}. Since the generators of $G_1$ are mapped by $\phi$ precisely in the set $\{R_\alpha \, |\, \alpha \in \Delta'\}$, we see that $\phi(G_1) < W(D_5)$ is isomorphic to $W(D_4)$.

    Now we prove that $\phi$ is injective on $G_3$. The group $\langle R_{\alpha_4}, R_{\alpha_5} \rangle$ is clearly isomorphic to $\mathbb{Z}/2\mathbb{Z} \times \mathbb{Z}/2\mathbb{Z}$. This implies that the order of the image of $\phi \colon \langle \gamma_4, \gamma_5 ,\gamma_1 \rangle \rightarrow \langle R_{\alpha_4}, R_{\alpha_5}, R_{\beta} \rangle$ is a multiple of $4$. Since $\langle \gamma_4,\gamma_5,\gamma_1\rangle \cong \mathfrak{S}_4$ has order $24$, its kernel is a normal subgroup of order 1,2,3 or 6. However, the only non-trivial normal subgroups of $\mathfrak{S}_4$ are the alternating subgroup (of order $12$) and the Klein group generated by ``double transpositions'' (of order $4$). Therefore the group $\langle R_{\alpha_4},R_{\alpha_5}, R_{\beta} \rangle=\phi(\langle \gamma_4, \gamma_5 ,\gamma_1\rangle)$ is isomorphic to $W(A_3)\cong \mathfrak{S_4}$. Moreover $\phi(\gamma_2)= R_{\alpha_2}$ commutes with the generators of this group. Since the centre of $\mathfrak{S}_4$ is trivial and $R_{\alpha_2}$ is non-trivial, it follows that $R_{\alpha_2}$ does not belong to $\phi(\langle \gamma_4, \gamma_5 ,\gamma_1\rangle)$, and therefore $\phi(G_3)\cong W(A_1)\times W(A_3)$.

    The fact that $\phi$ is injective on $G_4$ is checked similarly. Due to the injectivity on $G_1$, $\phi$ is injective on $\langle \gamma_2, \gamma_3, \gamma_5 \rangle\cong \mathfrak{S}_4$, which is a group of order $24$. This implies that the order of the image of $\phi \colon G_4 \rightarrow \langle R_{\beta},R_{\alpha_2}, R_{\alpha_3}, R_{\alpha_5} \rangle$ is a multiple of $24$. Since $120/24=5$, its kernel is a normal subgroup of $\mathfrak{S}_5$ of order $1$ or $5$. If this kernel were to be non-trivial, it would have to be a cyclic group generated by a $5$-cycle. However such a group is never a normal subgroup of $\mathfrak{S}_5$. It follows that the kernel is trivial, i.e. $\phi$ is injective on $G_4$. The very same argument applies to show the injectivity of $\phi$ on $G_5$.

    We have shown that $\Gamma_1$ is torsion free, and the the proof is compete.
\end{proof}

As a consequence of Proposition \ref{prop:phi and its kernel}, the quotient $M_1=\mathbb{H}^4/\Gamma_1$ is an orientable hyperbolic cusped $4$-manifold  with $\chi(M_1) = |W(D_5)| \cdot \chi(S) = 1920/960 = 2$.

\subsection{Parametrising the cusps}
We now prove that $M_1$ has exactly $10$ cusps. The automorphism of the regular orbifold cover $p \colon M_1 \rightarrow S$ associated to $\phi$ satisfies
$$\Aut(p) \cong \Gamma_S/\Gamma_1 \cong W(D_5).$$ (Recall that $\ker(\phi) = \Gamma_1$.) The cusps of $M_1$ correspond to the preimages of the ideal vertex $v_2$ of $S$, and are in one-to-one correspondence with the left cosets of the subgroup $\phi(G_2)<W(D_5)$, where $G_2 = \langle \gamma_3,\gamma_4,\gamma_5,\gamma_1 \rangle$.

We claim that $\phi(G_2)$ is isomorphic to $W(D_4)$. The argument is similar to the one used in the proof of the injectivity of $\phi$ on $G_1$. Indeed, the set of vectors of $D_5$ whose second coordinate is equal to $0$ forms a copy of the $D_4$ root system in the vector space $\{ x_2 = 0 \} \subset \matR^5$. A set of simple roots for this root system is given by $\Delta''=\{\delta, \alpha_3, \alpha_4, \alpha_5\}$ with $\delta=e_1-e_3$. The reflections $R_{\alpha_3}$, $R_{\alpha_4}$ and $R_{\alpha_5}$ are images of the generators of $G_2$, and thus belong to $\phi(G_2)$. Moreover, since $R_{\alpha_3}  R_{\alpha_4}  \beta = \delta$, by setting $B=R_{\alpha_3}  R_{\alpha_4}$ we see that $R_{\delta}=B  R_{\beta} B^{-1}$ belongs to $\phi(G_2)$, and this group is therefore isomorphic to $W(D_4)$.

From the above discussion, we see that 
$$\phi(G_2)=\{ A \in W(D_5) \, |\,  A e_2=e_2\}.$$
This provides us with a very convenient way to check whether or not two matrices in $W(D_5)$ belong to the same left coset for $\phi(G_2)$: given $A,B \in W(D_5)$, we have $$A  \phi(G_2)=B   \phi(G_2) \quad \Leftrightarrow \quad B^{-1}  A \in \phi(G_2) \quad \Leftrightarrow \quad B^{-1}  A e_2=e_2 \quad \Leftrightarrow \quad A  e_2 = B  e_2, $$
i.e. $A$ and $B$ belong to the same coset for $\phi(G_2)$ if and only if their second columns coincide. There are exactly $10$ cosets, since any column of a matrix of $W(D_5)$ corresponds to a vector of the form $\pm e_i$ for $i=1,\dots,5$ and the cusps of $M_1$ are naturally labelled by vectors of this form.

We now wish to understand how an automorphism of the orbifold covering $p \colon M_1 \rightarrow S$ acts on the cusps of $M_1$. The {action of $\Aut(p) \cong W(D_5)$} on the set of cusps of $M_1$ is induced by the action by left multiplication of $W(D_5)$ on the left cosets of $\phi(G_2)$. Now, if $A$ and $B$ belong to $W(D_5)$ and $B  e_2= \pm e_i$, we have that
$$(A  B)  e_2 = A  (B  e_2)= A   (\pm e_i).$$
Thus, the action on the cusps is determined by the action of $W(D_5)$ on $\{ \pm e_1, \ldots, \pm e_5 \}$.

\subsection{Pairing the cusps} 
We now build a convex projective manifold $\overline{M}_1$ with totally geodesic boundary by ``truncating the cusps'' of the hyperbolic manifold $M_1$. More precisely, we extend the homomorphism $\phi \colon \Gamma_S \rightarrow W(D_5)$ to $\overline{\phi} \colon \Gamma_P \rightarrow W(D_5)$ by mapping the extra generator $\gamma_0$ to the identity. The resulting convex projective orbifold $$\overline{M}_1=\Omega/\mathrm{ker(\overline{\phi})}$$ is in fact a ``manifold with mirror boundary'', that is an orbifold whose underlying space is a manifold with boundary and whose singular locus is the boundary. The latter consists of $10$ pairwise isomorphic totally geodesic components, corresponding to the cusps of $M_1$. Indeed, the non-trivial torsion elements in $\mathrm{ker}(\overline{\phi})$ are precisely the conjugates of $\gamma_0$. Each boundary component is isomorphic to the flat manifold corresponding to the kernel of $\phi|_{G_2}$.

We obtain a covering of convex projective orbifolds $\overline{p} \colon \overline{M}_1 \rightarrow P$ whose automorphism group is isomorphic to $W(D_5)$. In order to obtain a closed convex projective manifold $M$ which covers $P$, we will use some of these automorphisms to identify the boundary components of $\overline{M}_1$ in pairs. We shall choose some order-two elements of $\mathrm{Aut}(\overline{p})\cong W(D_5)$. In order for the resulting manifold to be orientable, we are forced to choose orientation-reversing elements, i.e. matrices of $W(D_5)$ with negative determinant. With this in mind, we proceed as follows.

\begin{itemize}
\item The matrix 
$C=\begin{pmatrix} 0 & 1 & 0 & 0 & 0\\
                   1 & 0 & 0 & 0 & 0\\
                   0 & 0 & -1 & 0 & 0\\
                   0 & 0 & 0 & -1 & 0\\
                   0 & 0 & 0 & 0 & 1
    \end{pmatrix}$ belongs to $W(D_5)$, has order $2$ and determinant $-1$. By looking at the images of vectors of the form $\pm e_i$, we see that the corresponding automorphism exchanges the boundary components of $\overline{M}_1$ as follows: 
                   $$\pm e_1 \leftrightarrow \pm e_2, \qquad \e_3 \leftrightarrow -e_3, \qquad e_4 \leftrightarrow -e_4, \qquad \pm e_5 \leftrightarrow \pm e_5.$$
\item The matrix 
$D=\begin{pmatrix} 0 & 1 & 0 & 0 & 0\\
                   1 & 0 & 0 & 0 & 0\\
                   0 & 0 & 1 & 0 & 0\\
                   0 & 0 & 0 & -1 & 0\\
                   0 & 0 & 0 & 0 & -1\end{pmatrix}$ belongs to $W(D_5)$, has order $2$ and determinant $-1$. The corresponding automorphism exchanges the boundary components of $\overline{M}_1$ as follows: 
                   $$\pm e_1 \leftrightarrow \pm e_2, \qquad \pm e_3 \leftrightarrow \pm e_3, \qquad e_4 \leftrightarrow -e_4, \qquad e_5 \leftrightarrow -e_5.$$
\end{itemize}

We use the matrices above to pair the boundary components of $\overline{M}_1$ via the corresponding automorphisms as follows. We use $C$ to pair the boundary components with labels $\pm e_1$ with those with labels $\pm e_2$, and those with label $e_i$ to those with label $-e_i$, for $i=3,4$. Notice that $C$ preserves the boundary components corresponding to $\pm e_5$. In order to pair these remaining two boundary components we use $D$, which indeed exchanges them. 

The resulting closed orientable convex projective manifold $M = \Omega/\Gamma$ covers the orbifold $P = \Omega/\Gamma_P$ with index 1920 by construction. The proof of Proposition \ref{prop:subgroup} (and so of Theorems \ref{thgm:onepiece} and \ref{thm:construction}) is complete.

\begin{rem}\label{rem:not-normal-cover}
    The covering $M\rightarrow P$ is not regular. Essentially, this is due to the fact that we are choosing two distinct elements 
    of $W(D_5) \cong \mathrm{Aut}(\overline{p})$ to pair the boundary components of $\overline{M}_1$. Even more, it can be quickly checked with a computer that there is no torsion-free normal subgroup of $\Gamma_P$ of index $\leq 1920$. 
\end{rem}

\appendix

\section{Rigidity and Superrigidity} \label{sec:rigidity} 

In Section \ref{sec:spazi-simmetrici} we make a heavy use of the classical results on the rigidity of lattices in semisimple Lie groups due to Mostow, Prasad, Margulis (and Corlette). We state them here in the form that we use. In what follows, a connected algebraic group $\mathbf{H}$ defined over a field $k\subset \mathbb{C}$ is \emph{adjoint} if the adjoint action of the complex Lie group $\mathbf{H}(\mathbb{C})$ on its complex Lie algebra is faithful.

\begin{thm}[Superrigidity]\label{teo:superrigidity}
    Let $G$ be a connected semisimple Lie group with trivial centre and no compact factors, $\Gamma<G$ an irreducible lattice, $\mathbf{H}$ a connected, adjoint, real algebraic group such that $\mathbf{H}(\matR)^{\circ}$ has no compact factors, and $\delta\colon \Gamma \rightarrow \mathbf{H}(\matR)$ a non-trivial group homomorphism with Zariski-dense image. 
    
    If either $G$ has real rank $\geq 2$ or $G$ is isomorphic to $\mathrm{PSp}(n,1)$ or the exceptional adjoint group $G_4^{-20}$, then
    $\phi$ extends to a continuous surjective homomorphism $\phi\colon G\rightarrow \mathbf{H}(\matR)^{\circ}$. If moreover $\phi(\Gamma)$ is discrete in $\mathbf{H}(\matR)^{\circ}$, then $\phi$ is an isomorphism.
\end{thm}

The case where $G$ has real rank $\geq 2$ is essentially one of the formulations of the celebrated Margulis' Superrigidity Theorem \cite[Theorem 6.16, Remark 6.17]{M}. Its extension to the case of the rank-one groups $\mathrm{PSp}(n,1)$ and $F_4^{-20}$ is a subsequent result by Corlette \cite{Cor}.

By combining the Superrigidity Theorem for lattices in semisimple Lie groups of rank $\geq 2$ with the Mostow--Prasad rigidity for lattices in rank-one simple Lie groups we obtain the following:
\begin{thm}[Margulis, Mostow, Prasad]\label{teo:Mostow} 
    Let $G_1$ and $G_2$ be connected, semisimple Lie groups with trivial centre, no compact factors and no factor isomorphic to $\mathrm{SO}^+_{2,1}(\matR)\cong \mathrm{PSL}_2(\matR)$, and $\Gamma_1<G_1$, $\Gamma_2<G_2$ two lattices. Then any isomorphism $\delta\colon \Gamma_1 \rightarrow \Gamma_2$ extends to a continuous isomorphism $\phi\colon G_1 \rightarrow G_2$.
\end{thm}

The only reason for excluding factors isomorphic to $\mathrm{PSL}_2(\matR)$ is the fact that a surface group $\Gamma$ admits uncountably many non $\mathrm{PGL}_2(\matR)$-conjugate discrete and faithful respresentations in $\mathrm{PSL}_2(\matR)$. Therefore in this case the two inclusions of $\Gamma$ in $G_1=\mathrm{PSL}_2(\matR)$ and $G_2=\mathrm{PSL}_2(\matR)$ may not induce a continuous isomorphism from $G_1$ to $G_2$, but it still holds true that $G_1$ and $G_2$ are isomorphic. 

Indeed the following result still holds true \cite[Theorem 15.1.1]{Morris}:
\begin{cor}\label{cor:weak-Mostow}
Let $G_1$ and $G_2$ be connected, semisimple Lie groups with trivial centre and no compact factors, and $\Gamma_1<G_1$, $\Gamma_2<G_2$ two lattices. If $\Gamma_1$ is isomorphic to $\Gamma_2$, then $G_1$ is isomorphic to $G_2$.\end{cor}

We conclude with the proof of a technical lemma:
\begin{lemma}\label{lemma:reducibility}
    Let $G=G' \times D$ be a product of a semisimple Lie group of non-compact type $G'$ and an infinite discrete group $D$. Then any lattice $\Gamma<G$ is reducible, in the sense that it has a finite-index subgroup of the form $\Gamma_1 \times \Gamma_2$, where $\Gamma_1<G'$ and $\Gamma_2 < D$ are lattices.
\end{lemma}

\begin{proof}
    Denote by $p_1$ and $p_2$ the projections of $G$ onto $G'$ and $D$ respectively. The coset space $G/\Gamma$ inherits a foliation from the foliation of $G$ into copies of $G'$. Each leaf is isomorphic to the coset space $G'/(\Gamma \cap G')$, while the leaves are in a one-to-one correspondence with the discrete coset space $D/p_2(\Gamma)$. 
    
    Since $\Gamma$ is a lattice, $\Gamma \cap G'$ is necessarily a lattice in $G'$, and $p_2(\Gamma)$ is a finite-index subgroup of $D$. Moreover the group $p_1(\Gamma)<G'$ normalises $\Gamma \cap G'$. Since in a semisimple Lie group of non-compact type every lattice has finite index in its normaliser, $p_1(\Gamma)$ is a lattice.

    Now, the finite volume coset space $G/\Gamma$ is also foliated into copies of $D/(\Gamma\cap D)$, and these are in a one-to-one correspondence with the coset space $G'/p_1(\Gamma)$. Thus $\Gamma \cap D$ is a finite-index subgroup of $D$, and therefore the group $(\Gamma \cap G') \times (\Gamma \cap D)<\Gamma$ is a lattice in $G$.
\end{proof}


\begin{thebibliography}{99}




\bibitem{A} \textsc{A.~Avez}, Remarque sur les formes de Pontrjagin. \emph{C.~R.~Acad.~Sci., Paris, S\'er.~A} \textbf{270} (1970), 1248.

\bibitem{BDL} \textsc{S.~A.~Ballas, J.~Danciger, G.-S.~Lee}, Convex projective structures on nonhyperbolic three-manifolds. \emph{Geom.~Topol.} \textbf{22} (2018), 1593--1646.


\bibitem{B2} \textsc{Y.~Benoist}, Convexes divisibles II, \emph{Duke~Math.~J.} \textbf{120} (2003), 97–-120.

\bibitem{B4} \bysame, Convexes divisibles IV. \emph{Invent.~Math.} \textbf{164} (2006), 249--278.

\bibitem{Bqi} \bysame, Convexes hyperboliques et quasiisom\'etries. \emph{Geom.~Dedicata} \textbf{122} (2006), 109--134.

\bibitem{Bquest} \bysame, Exercises on divisible convex sets. \href{https://www.imo.universite-paris-saclay.fr/~yves.benoist/prepubli/12GearJuniorRetreat.pdf}{\url{https://www.imo.universite-paris-saclay.fr/~yves.benoist/prepubli/12GearJuniorRetreat.pdf}} (2012, updated 2023).


\bibitem{BV} \textsc{P.-L.~Blayac, G.~Viaggi}, Divisible convex sets with properly embedded cones. \emph{Publ.~Math.~IHES} (2024).

\bibitem{B} \textsc{M.~D.~Bobb}, Codimension-1 simplices in divisible convex domains. \emph{Geom.~Topol.} \textbf{25} (2021), 3725--3753.


\bibitem{Borel-density} \textsc{A.~Borel}, Density properties for certain subgroups of semisimple groups without compact components. \emph{Ann.~Math.} \textbf{72} (1960), 179–-188.




\bibitem{CLM2} \textsc{S.~Choi, G.-S.~Lee, L.~Marquis}, Convex projective generalized Dehn filling. \emph{Ann.~Sci.~\'Ec.~Norm.~Sup\'er.} \textbf{53} (2020) 217--266.

\bibitem{CLM1} \bysame, Deformation spaces of Coxeter truncation polytopes. \emph{J.~London~Math.~Soc.} \textbf{106}:4 (2022), 3822--3864.

\bibitem{C} \textsc{R.~Chow}, Groups quasi-isometric to complex hyperbolic space. \emph{Trans.~Am.~Math.~Soc.} \textbf{348} (1996), 1757--1769.

\bibitem{CM} \textsc{M.~Conder, C.~Maclachlan}, Small volume compact hyperbolic $4$-manifolds. \emph{Proc.~Amer.~Math.~Soc.} \textbf{133} (2005), 2469--2476.

\bibitem{Cor} \textsc{K.~Corlette}, Archimedean Superrigidity and Hyperbolic Geometry. \emph{Ann.~Math.}, \textbf{135} (1992), 165--182.

\bibitem{DFWZ} \textsc{S.~Douba,~B.~Fl\'echelles, T.~Weisman, F.~Zhu}, Cubulated hyperbolic groups admit Anosov representations. To appear in \emph{Geom.~Topol.}, \href{https://arxiv.org/abs/2309.03695}{\texttt{arXiv:2309.03695}}. 

\bibitem{EB1} \textsc{P.~Eberlein}, A canonical form for compact non-positively curved manifolds whose fundamental groups have nontrivial center. \emph{Math.~Ann.} \textbf{260} (1982), 23--29.

\bibitem{EB2} \bysame, Euclidean de Rham factor of a lattice of nonpositive curvature. \emph{J.~Diff.~Geom.} \textbf{18} (1983), 209--220. 

\bibitem{EB3} \bysame, L-subgroups in spaces of nonpositive curvature. ``Curvature and topology of Riemannian manifolds'', \emph{Lect.~Notes~Math.} \textbf{1201} (1986), 41--88.

\bibitem{FJ1} \textsc{F.~T.~Farrell, L.~E.~Jones}, Topological rigidity for compact nonpositively curved manifolds. \emph{Proc.~Sympos.~Pure~Math.} \textbf{54} (1993), 229--274.

\bibitem{FJ2} \bysame, 
Rigidity for aspherical manifolds with $\pi_1(M)\subset \mathrm{GL}_m(\mathbb{R})^*$. \emph{Asian~J.~Math.} \textbf{2} (1998), 215--262.

\bibitem{F} \textsc{R.~O.~Filipkiewicz}, \emph{Four-Dimensional Geometries}. Ph.D thesis, University of Warwick (1984).

\bibitem{Geng} \textsc{A. Geng}, \emph {The classification of five-dimensional geometries}, ProQuest LLC, Ann Arbor, MI (2016).




\bibitem{Helgason} \textsc{S.~Helgason}, Differential Geometry, Lie Groups, and Symmetric Spaces, American Mathematical Society, (2001).

\bibitem{Hpaper} \textsc{J.~A.~Hillman} On 4-manifolds which admit geometric decompositions. \emph{J.~Math.~Soc.~Japan} \textbf{50} (1998), 415--431. 

\bibitem{H} \bysame, Four-manifolds, geometries and knots. \emph{Geometry~and~Topology~Monographs} \textbf{5} (2002).



\bibitem{IZ2} \textsc{M.~Islam, A.~Zimmer}, A flat torus theorem for convex co-compact actions of projective linear groups. \emph{J.~Lond.~Math.~Soc.~(2)} \textbf{3} (2021), 470--489.

\bibitem{IZ} \bysame, Convex cocompact representations of 3-manifold groups. \emph{J.~Topol.} \textbf{17} (2024).


\bibitem{K} \textsc{M.~Kapovich}, Convex projective structures on Gromov-Thurston manifolds. \emph{Geom.~Topol.} \textbf{11} (2007), 1777--1830.

\bibitem{Kob} \textsc{S.~Kobayashi}, Projectively invariant distances for affine and projective structures. ``Differential geometry''
, \emph{Banach~Cent.~Publ.} \textbf{12} (1984), 127--152.

\bibitem{Koe} \textsc{M.~Koecher}, The Minnesota notes on Jordan algebras and their applications. \emph{Lecture Notes in Mathematics} \textbf{1710} (1999), Springer-Verlag.

\bibitem{KS} \textsc{A.~Kolpakov, L.~Slavich}, Symmetries of hyperbolic 4-manifolds. \emph{Int.~Math.~Res.~Not.~IMRN} \textbf{2016} (2016), 2677--2716.



\bibitem{LR2} \textsc{J.-F.~Lafont, L.~Ruffoni}, Special cubulation of strict hyperbolization. \emph{Invent.~Math.} {\textbf 236} (2024), 925--997.

\bibitem{LMR} \textsc{G.-S.~Lee, L.~Marquis, S.~Riolo}, A small closed convex projective 4-manifold via Dehn filling. \emph{Publ.~Mat.} \textbf{66}:1 (2022), 369--403.



\bibitem{L} \textsc{C.~Long}, Small volume closed hyperbolic $4$-manifolds. \emph{Bull.~London~Math.~Soc.} \textbf{40} (2008), 913--916.

\bibitem{M} \textsc{G.~Margulis}, Discrete subgroups of semisimple Lie groups. Springer-Verlag (1991).



\bibitem{Morris} \textsc{D.~W.~Morris}, An introduction to Arithmetic groups. Deductive press (2015).

\bibitem{N} \textsc{S.~P.~Novikov}, Rational Pontrjagin classes. Homeomorphism and homotopy type of closed manifolds. I. \emph{Izv.~Akad.~Nauk~SSSR~Ser.~Mat.} \textbf{29} (1965), 1373--1388.


\bibitem{O} \textsc{P.~Ontaneda}, Riemannian hyperbolization. \emph{Publ.~Math.~IHES} \textbf{131} (2020), 1--72.


\bibitem{P} \textsc{P.~Pansu}, Metriques de Carnot-Caratheodory et quasi-isometries des espaces de rang 1. \emph{Ann.~Math.} \textbf{129} (1989), 1--60.

\bibitem{PR} \textsc{G.~Prasad, M.~S.~Ragunathan}, Cartan subgroups and lattices in semi-simple groups,\emph{Ann.~Math.} \textbf{96(2)} (1972), 296--317. 

\bibitem{Rag} \textsc{M.~S.~Raghunathan}, Discrete Subgroups of Lie Groups. \emph{Springer-Verlag} (1972).

\bibitem{RT} \textsc{J.~G.~Ratcliffe, S.~T.~Tschantz}, The volume spectrum of hyperbolic 4-manifolds. \emph{Experiment.~Math.} \textbf{9} (2000), 101--126.

\bibitem{R} \textsc{L.~Ruffoni}, Manifolds without real projective or flat conformal structures. \emph{Proc.~Amer.~Math.~Soc.} {\bf 151} (2023), 3611--3620.


\bibitem{Tbook} \textsc{W.~P.~Thurston}, Three-dimensional geometry and topology, vol. 1 \emph{Princeton University Press} (1997).

\bibitem{T} \textsc{B.~Tshishiku}, Pontryagin classes of locally symmetric manifolds. \emph{Algebr.~Geom.~Topol.} \textbf{15} (2015), 2709--2756.

\bibitem{Vey} \textsc{J.~Vey}, Sur les automorphismes affines des convexes ouvertes saillants. \emph{Ann.~Sc.~Norm.~Sup.~Pisa} \textbf{24} (1970), 641--665.

\bibitem{V} \textsc{\'E.~B.~Vinberg} The theory of homogeneous convex cones. \emph{Trudy~Moskov.~Mat.~Ob\v{s}\v{c}.} \textsc{12} (1963), 303--358.

\bibitem{V2} \bysame, The structure of the group of automorphisms of a homogeneous convex cone. \emph{Trudy~Moskov.~Mat.~Ob\v{s}\v{c}.} \textsc{13} (1965), 63--93.

\bibitem{V3} \bysame, Discrete linear groups that are generated by reflections.  \emph{Izv.~Akad.~Nauk~SSSR~Ser.~Mat.} \textsc{35} (1971), 1072--1112.

 
\end{thebibliography}
\end{document}